\documentclass[final]{siamltex}
\usepackage{amsmath}
\usepackage{mathtools}
\usepackage{mathrsfs}
\usepackage{amssymb}
\usepackage{amsfonts}
\usepackage{amsxtra}
\usepackage{amstext}
\usepackage{amsbsy}
\usepackage{amscd}
\usepackage{graphicx}
\usepackage{float}
\usepackage{cite}
\usepackage{xcolor}
\usepackage{srcltx} 
\usepackage{marginnote,slashbox}
\usepackage{rotating}
\definecolor{darkblue}{rgb}{0.0,0.0,0.6}
\definecolor{darkgreen}{rgb}{0.0,0.6,0.0}
\usepackage[pdftex,colorlinks=true,urlcolor=darkblue,citecolor=darkblue,linkcolor=darkblue]{hyperref}
\usepackage[utf8]{inputenc}      
\usepackage{booktabs}            
\usepackage{url}                 
\usepackage[T1]{fontenc}         
\usepackage{algorithmic}         
\usepackage[pagewise]{lineno}
\usepackage{calc}
\usepackage[notcite,notref]{showkeys}

\numberwithin{table}{section}    
\numberwithin{figure}{section}   
\numberwithin{equation}{section} 
\usepackage{my_latex_commands}
\newtheorem{assumption}[theorem]{Assumption}
\newtheorem{remark}[theorem]{Remark}

\newcommand{\ro}[1]{\textcolor{black}{#1}}

\newcommand{\ff}{\forall\,}

\newcommand{\eps}{\varepsilon}
\newcommand{\e}{\varepsilon}

\newcommand{\ld}{L^2(D)}

\newcommand{\hde}{H^s(D\setminus \bar E)}

\newcommand{\hd}{H^1_0(D)}
\newcommand{\hdd}{H^1(D)}

\newcommand{\gbe}{\bar g_\eps}
\newcommand{\hge}{\widetilde {g_\eps}}

\newcommand{\li}{L^\infty(D)}

\newcommand{\dense}{\overset{d}{\embed}}
\renewcommand{\o}{\omega}

\newcommand{\ogg}{\O_{g}}

\begin{document}

\title{Approximation of   shape optimization  problems with non-smooth PDE constraints:  an optimality conditions point of view}
\date{\today}
\author{L.\,Betz\footnotemark[1]}
\renewcommand{\thefootnote}{\fnsymbol{footnote}}
\footnotetext[1]{Faculty of Mathematics, University of Wuerzburg,  Germany}
\renewcommand{\thefootnote}{\arabic{footnote}}

\maketitle
\begin{abstract}
This paper is concerned with a shape optimization problem governed by a non-smooth PDE, i.e., the nonlinearity in the state equation is not necessarily differentiable. 
We follow the functional variational approach of \cite{ pen} where the set of admissible shapes is parametrized by a large class of continuous mappings. This methodology allows for both boundary and topological variations. It has the advantage that one can  rewrite the shape optimization problem as a control problem in a function space. To overcome the lack of convexity of the set of admissible controls, we provide an essential density property. This permits us to show that each parametrization associated to the optimal shape is  the limit of global optima 
of  non-smooth distributed optimal control problems. The admissible set  of the approximating minimization problems is a convex subset of a Hilbert space of functions. Moreover, its  structure is such that one can derive strong stationary optimality conditions \cite{p2}. The present manuscript provides the basis for the investigations from \cite{p3}, where   necessary  conditions in form of an  optimality system have been recently established.
\end{abstract}


  \section{Introduction}
  Shape optimization problems  arise in many engineering applications \cite{dz} and  are often formulated as minimization problems governed by one or more PDEs or VIs \cite{sz}. Their main particularity is the fact that these  equations are solved on  unknown domains \cite{sz, pir}. The  goal is to find the optimal shape, i.e., that domain for which the distance to a certain desired state is minimized.
 Thus, shape optimization problems exhibit similarities  with optimal control problems \cite{troe}, the essential difference and difficulty being that the admissible control set consists of variable geometries.
Such problems are highly nonconvex, which makes  their investigation  challenging both from the theoretical and the computational point of view.

  
  There is an enormous  amount of literature devoted to  the study of optimal design  problems, see e.g.\,the classical contributions \cite{sz,pir} and the references therein. To keep the depiction concise, we just focus on the works that deal with these problems at a theoretical level, with an emphasis on optimality systems in qualified  form. 
These resemble the classical Karush-Kuhn-Tucker conditions, and in the case  of optimal control problems, they involve an adjoint equation. In \cite{zs}, topology and  boundary variations are combined for the first time in order to derive necessary optimality conditions.
In \cite{at}, the considered shape and topology optimization problem is governed by a cone constraint and the existence of Lagrange multipliers is shown, the respective optimality conditions being expressed in the form of a complementarity system. 
Methods for computing the shape derivative of the cost functional without
resorting to the differentiability of the control-to-state map have been developed in \cite{k_sturm,kun, kk}. The techniques therein are applied for shape 
optimization problems governed by smooth PDEs.

Shape optimization problems with non-smooth constraints have been investigated at a theoretical level mostly 
with respect to existence of optimal shapes \cite{dm1, dm2} and   sensitivity analysis \cite{ nsz, hs,sz, sok}.  While there is an increasing number of contributions   concerning  optimal shape design problems governed by VIs, see \cite{dm1,dm2, nsz, hs, ylw, lsw} and the references therein, there are no papers known to the author that address the case where the governing equation is a \textit{non-smooth PDE}. In \cite{ylw, lsw}, the authors   resort to smoothening techniques and optimality systems in qualified form are obtained just for the smoothened problem \cite{ylw}, or, for the original problem \cite{lsw}, but  only if certain assumptions are imposed on the converging sequences. If smoothening is not involved, optimality conditions for the non-smooth shape optimization problem  do not involve an adjoint equation, unless linearity w.r.t.\,direction is assumed \cite{nsz}. Otherwise, these are stated just in a primal form \cite{hs}, i.e., the respective optimality condition only asserts the non-negativity of the shape derivative of the reduced objective functional in feasible directions.

As in most of the literature, the  approaches in all the aforementioned  papers (see also the references therein) are based on variations of the geometry. One of the  most common notions in this context are the 
shape derivative \cite{sz} and the topological derivative \cite{top_book}.

A more novel technique to derive optimality conditions in qualified form, where general functional variations instead of geometrical ones are involved, can be found in \cite{oc_t}. Therein, an optimal design problem governed by a linear PDE with Neumann boundary conditions is investigated. By means of the implicit parametrization theorem \cite{impl, impl3} combined with  Hamiltonian systems \cite{it_h},  the equivalence of the shape and topology optimization problem with an optimal control problem in function space is established, provided that the set of admissible geometries is generated by a certain class of continuous functions. The respective control problem in function space  is amenable to the derivation of optimality conditions by a  classical Lagrange multipliers approach. 
The same functional variational method will be adopted (to some extent) in the present work.


The aim of this paper is to provide a first essential step towards the derivation of  optimality systems  for the optimal shape associated to the following \textit{non-smooth} {shape} optimization problem 
\begin{equation}\tag{$P_\O$}\label{p_sh}
 \left.
 \begin{aligned}
  \min_{\O \in \OO, E \subset \O} \quad & \int_E (y_{\O}(x)-y_d(x))^2 \; dx+\alpha\,\int_{\O}\,dx, 
\\     \text{s.t.} \quad & 
  \begin{aligned}[t]
    -\laplace y_{\O} + \beta(y_{\O})&=f \quad \text{in }\O,
   \\y_{\O}&=0  \quad \text{on } \partial \O.\end{aligned} \end{aligned}
 \quad \right\}
\end{equation}
The  main particularities of \eqref{p_sh} are:
\begin{itemize}
\item the non-smooth character of $\beta$, which is supposed to be locally Lipschitz continuous and directionally differentiable, but \textit{not} necessarily \textit{differentiable}, e.g., $\beta=\max\{\cdot,0\}$ (we underline that we do not intend to replace $\beta$ by a differentiable function);
\item the fact that the governing PDE is solved on the unknown (variable) domain $\O$ (which plays the role of the control).
\end{itemize}
The admissible control set $\OO$ consists of (unknown) subdomains of a given, fixed domain $D$. These are the so-called \textit{admissible shapes}. In this manuscript, they are generated by a class of continuous functions, see Definition \ref{def:o}. \ro{As it turns out, $\OO$ covers all  subdomains of $D$ of class $C^2$ that contain the  observation set  $E$ and whose boundaries do not touch $\partial D$, cf.\,Proposition \ref{prop:c2}.} The holdall domain $D\subset \R^2$ is bounded, of class $C^{1,1}$, and $f \in L^2(D)$. In the current work, the non-smoothness $\beta$ can be thought of as a piecewise linear mapping. In \cite{p2}, which specifically deals with the lack of differentiability (see below), we will consider much more general situations. The symbol $-\laplace: \hd \to H^{-1}(D)$ denotes the Laplace operator in the distributional sense; note that $\hd$  is the closure of the set {$ C^\infty_c(D) $} w.r.t.\,the $H^1(D)-$norm.
The desired state $y_d$ is an $L^2-$function, which is defined \ro{on a fixed subdomain of $D$}, that  is, the  observation set  $E$ (Assumption \ref{assu:E}). The parameter $\alpha$ appearing in the objective is supposed to satisfy $\alpha \geq 0$.  

The main purpose of this manuscript is to develop an approximation scheme for \eqref{p_sh} by means of more approachable optimal control problems  where the underlying control space consists of functions, see \eqref{p10} below. We show that each parametrization associated to the optimal shape of \eqref{p_sh} is the limit of global minimizers of  \eqref{p10} (Corollary \ref{cor:os}). The approximating control problem is governed by a non-smooth PDE in the holdall domain $D$. Its structure  is such that, under certain requirements on the given data, it permits the derivation of strong stationary optimality systems \cite{p2}, i.e., systems that are equivalent to the  first-order necessary optimality condition. In the recent work \cite{p3} we to carry out  the passage to the limit $\e \to 0$ in these optimality conditions.

The starting point of our analysis is the functional variational method  introduced by  \cite{nt}, where the admissible set $\OO$ is generated by a large class of functions \cite{pen}. We point out that this is  different from the level set method of \cite{os}. 
It allows one to switch from the shape optimization problem to a control problem in a function space, the latter being more amenable to a rigorous mathematical investigation.
 This technique involves the  description of the boundary of the unknown domain as the solution to a Hamiltonian system and it is based on the implicit function theorem \cite{impl, impl3} and Poincar\'e-Bendixson theory \cite[Ch.\,10]{hsd}.  The idea that the admissible shapes are parametrized by so called shape functions turned out to be successful  in various papers \cite{pen, t_jde, oc_t, nt, nt2, to, mt_new2, mt_new1, hmt}.
 The variations used in all these works have no prescribed geometric form (as usual in the literature) and the  methodology therein provides a unified analytic framework allowing for both boundary and topological variations. 
 Penalization methods were developed in \cite{pen} and extended in \cite{ t_jde, to}. We mention in particular \cite{t_jde} where the shape and topology optimization problem governed by a PDE with Neumann boundary conditions is equivalently rewritten as a control problem in a function space where the state equation is solved on the fixed domain while the boundary condition appears penalized in the objective.
The functional variational approach was also the fundament for the derivation of optimality conditions for optimal shape design problems governed by linear PDEs with Neumann boundary conditions in \cite{oc_t}. We also refer to the more recent contributions \cite{mt_new1} (clamped plates) and \cite{mt_new2} (Navier-Stokes equations).

Given a continuous function $g$  on the holdall domain $D$, we define 
 \begin{equation}\label{def:og0}
 \Omega_g:=\interior\{x \in D: g(x) \leq 0\}.
\end{equation}
To ensure that $E \subset \O_g$, one requires \[g \ro{<} 0 \quad \text{in }E.\]  
Two other key properties which need to be imposed on the parametrization (of the unknown domain) are 
 \[|\nabla g|+|g|>0 \quad \text{in }D, \qquad g>0 \quad \text{on }\partial D.\]If $g \in C^2(\bar D)$, these  imply that $\partial \O_g$ is a finite union of closed disjoint $C^2$ curves, without self intersections, that do not intersect $\partial D$ \cite[Prop.\,2]{pen}. \ro{Let us stress out that this result in valid in two dimensions only and this is  the reason why we stick to the two-dimensional framework in this paper.} The above conditions imposed on $g$ give rise to the definition of the set of admissible shape functions, called $\FF_{\mathfrak{s}}$, see \eqref{f_s} and Remark \ref{rem:f_s} below.
We note that $\O_g$ may have many connected components. The admissible shape (domain)  that we use in the definition of $\OO$ below, see \eqref{o}, is the \ro{connected} component that contains the subdomain $E$. Since this may not be simply connected, the approach we discuss in this paper is related to topological optimization too.

By proceeding like this, we arrive at a reformulation of \eqref{p_sh} in terms of a  control problem in a function space. This reads as follows
\begin{equation}\tag{$P$}
 \left.
 \begin{aligned}
  \min_{g \in \FF_{\mathfrak{s}}} \quad & \int_E (y_{g}(x)-y_d(x))^2 \; dx+\alpha\,\int_{D}1-H(g)\,dx, 
\\     \text{s.t.} \quad & 
  \begin{aligned}[t]
    -\laplace y_{g} + \beta(y_{g})&=f \quad \text{in }\O_g,
   \\y_{g}&=0  \quad \text{on } \partial \O_g, \end{aligned} \end{aligned}
 \quad \right\}
\end{equation}where $H:\R \to\{0,1\}$ is the Heaviside function, see \eqref{h} below.

We underline the difficulty of the problem \eqref{p_shh}. The application of traditional optimization methods is excluded, as we deal with a control problem governed by a non-smooth PDE with the additional challenges  that the admissible set $\FF_{\mathfrak{s}}$ is non-convex, while the control does not appear on the right-hand side of the non-smooth PDE, but in the definition of the variable domain on which this is solved. Even in the case of classical control problems (where the domain is fixed), the non-smooth character is particularly  challenging, as the standard KKT theory cannot be directly employed if the differentiability of the control-to-state map is not available.
When we exclude the prominent smoothening techniques from \cite{barbu84}, 
  the derivation of (strong stationary) optimality conditions for non-smooth control problems in the presence of  control constraints  \cite{wachsm_2014, mcrf} is restricted to certain situations (so-called 'constraint qualifications'). 
As far as we know, there isn't any literature on the topic of optimality conditions for non-smooth optimal control problems with non-convex admissible set that does not involve smoothening. We emphasize that  we do not intend to  
resort to such techniques. In the  subsequent papers \cite{p2,p3} that continue the findings from the present manuscript,  all the  problems examined  (\eqref{p_sh}, \eqref{p_shh} and \eqref{p10}) feature the same non-differentiable mapping $\beta$.

  In order to deal with the control problem \eqref{p_shh}, and thus with  the shape optimization problem \eqref{p_sh}, 
  we introduce an approximation scheme  of fixed domain type \cite{nt, nt2}.  That is, we extend the state equation on the whole reference domain $D$.  
 We only smoothen the Heaviside function, since this has jumps, unlike the non-smoothness $\beta$ which is supposed to be continuous. By proceeding like this, we preserve the non-smooth character, and arrive, for $\e>0$ small, fixed, at the following approximating optimal control problem:
\begin{equation}\tag{$P_\eps$}
 \left.
 \begin{aligned}
  \min_{g \in   \FF} \quad & \int_E (y(x)-y_d(x))^2 \; dx+\alpha\,\int_{D} (1-H_\eps(g))(x)\,dx+\frac{1}{2}\,\|g-\bar g_{\mathfrak{s}}\|_{\WW}^2  \\   \text{s.t.} \quad & 
  \begin{aligned}[t]
   -\laplace y + \beta(y)+\frac{1}{\eps}H_\eps(g)y&=f +\eps g\quad \text{in }D,
   \\y&=0\quad \text{ on } \partial D. \end{aligned} 
 \end{aligned}
 \quad \right\}
\end{equation}
In the objective of \eqref{p10}, $\bar g_{\mathfrak{s}}$ is a local optimum of \eqref{p_shh} which  we intend to approximate by local optima of \eqref{p10} (adapted penalization \cite{barbu84}), in the sense that $\bar g_{\mathfrak{s}}$ will arise as the limit of a sequence of local minima of \eqref{p10} (Theorem \ref{thm:cor}). The mapping $H_\e$ is the regularization of the Heaviside function, cf.\,\eqref{reg_h} below.
We underline the fact that we replaced the non-convex set $\FF_{\mathfrak{s}}$ with a convex subset of the Hilbert-space $\WW:=\ld \cap \hde$, $s\in(1,2]$, namely
\begin{equation}
 \FF:=\{g \in \WW: g \leq 0\text{ a.e.\,in }E\}.
\end{equation}
The formulation of \eqref{p10} is slightly reminiscent of the minimization problem recently studied in \cite{mcrf} with the additional feature that the control enters the state equation in a non-linear fashion (via the mapping $H_\e$). However, the structure of \eqref{p10} allows us to state conditions on the given data under which the derivation of a complete (so-called strong stationary) optimality system is possible \cite{p2}.  
For more comments regarding the particularities of \eqref{p10}, see section \ref{sec:3}.



Our main goal is to show that local optima of \eqref{p_shh} can be approximated by local optima of  \eqref{p10} (Theorem \ref{thm:cor}). 
One of the fundamental challenges will arise from the fact that in \eqref{p10} we replaced the non-convex set $\FF_{\mathfrak{s}}$ from the original problem \eqref{p_shh} by $\FF$. To bridge this gap, we prove that $\FF_{\mathfrak{s}}$ is dense in $\FF$ w.r.t.\,the $\ld-$norm (section \ref{dense}).
Since  the control-to-state operator of \eqref{p_shh} is defined on $\FF_{\mathfrak{s}}$ only, not on $\FF$, the meanwhile standard arguments of the adapted penalization method  \cite{barbu84}  must be accordingly changed. This requires establishing certain convergences of the control-to-state operator $g \mapsto y_g$ of \eqref{p_shh} (subsection \ref{sec:s_e}), which are more difficult to obtain compared to the classical case where the domain is fixed and the control is distributed. The involved techniques, especially the density $\FF_{\mathfrak{s}} \dense \FF$, may be employed in the study of other similar problems (Remark \ref{rem:dense}).

The manuscript is organized as follows. After introducing the precise assumptions at the beginning of section \ref{sec:2}, subsection \ref{sec:as} deals with the definition of the admissible set $\OO.$ \ro{Here we establish that this consists of a certain class of $C^2$ domains  that contain the fixed subdomain $E$ (Proposition \ref{prop:c2}).} Then, in subsection \ref{sec:ref}, we reformulate the non-smooth shape optimization problem \eqref{p_sh} as a non-smooth control problem in a function  space, i.e., as  \eqref{p_shh}. Here, we also introduce the notion of local optimum for the latter. Section \ref{sec:3'} is dedicated to the  density property $\FF_{\mathfrak{s}} \dense \FF$, which is essential for making the connection between \eqref{p_shh} and the approximating non-smooth control problem with fixed domain  \eqref{p10}. This is introduced in section \ref{sec:3}. In subsection \ref{sec:solv} we show that the unique solution  of the state equation in \eqref{p_shh} is the limit of solutions to state equations in \eqref{p10}. Subsection \ref{sec:s_e} is dedicated to the convergence properties of the control-to-state maps of \eqref{p_shh} and \eqref{p10}. In subsection \ref{ex_p10} we establish that \eqref{p10} admits optimal solutions. Finally, section \ref{cor} contains the main results of this paper, namely Theorem \ref{thm:cor} and Corollary \ref{cor:os}. Theorem \ref{thm:cor} states that the approximation of \eqref{p_shh} by \eqref{p10} is meaningful in the sense that local minima of \eqref{p_shh} arise as limits of local optima of \eqref{p10}.  As a consequence, each parametrization of the optimal shape of \eqref{p_sh} is approximable by global minimizers of \eqref{p10} (Corollary \ref{cor:os}). The findings in this last section provide the basis for the investigations with respect to the derivation of qualified optimality systems for \eqref{p_sh}, see  \cite{p3}.

  \section{The shape optimization problem}\label{sec:2}
We begin this section by stating the precise assumptions on the non-smoothness $\beta$ appearing in the state equation in \eqref{p_sh}. We also mention the requirements needed for the observation set $E$. Then, in the upcoming subsections  we will introduce the set of admissible shapes $\OO$ for \eqref{p_sh}. This will allow us to 'rewrite' \eqref{p_sh} as an optimal control problem in a function space ({Proposition} \ref{rem:equiv}).
  \begin{assumption}[The non-smoothness]\label{assu:stand}
The  function  $\beta: \R \to \R$ is   
monotone increasing and {locally} Lipschitz continuous {in the following sense: F}or all $M>0$, there exists 
  a constant $L_M>0$ such that
\[
   |\beta(z_1) - \beta(z_2)| \leq L_M |z_1 - z_2| \quad \forall\, z_1, z_2 \in [-M,M] .
\]
 \end{assumption}
\begin{remark}
At this stage, we do not make any assumptions concerning the  limited differentiability properties of $\beta$ that were advertised in the introduction. These do not play any role in the present paper, but in {\cite{p2}}, which is a continuation of the present work. We just mention that, therein, $\beta$ is supposed to be semi-differentiable only,
i.e., the right and left side derivative of $\beta$ may not be equal at certain points, e.g.\, $\beta=\max\{\cdot,0\}$ or $\beta=|\cdot|.$
This is precisely one of the main challenges in \cite{p2}.
\end{remark}

By Assumption \ref{assu:stand}, it is straight forward to see that the Nemytskii operator $\beta:\li  \to \li$ is well-defined. Moreover, this is Lipschitz continuous on bounded sets in the following sense:  for every  $M > 0$, 
  there exists $L_M > 0$ so that
  \begin{equation}\label{eq:flip}
   \|\beta(y_1) - \beta(y_2)\|_{L^q(D)} \leq L_M \, \|y_1 - y_2\|_{L^q(D)} \quad \forall\, y_1,y_2 \in \clos{B_{\li}(0,M)},\ \forall\, 1\leq q \leq \infty.
  \end{equation}
 \begin{assumption}[The observation set]\label{assu:E}
The  set $E\subset D$  is a  domain. Moreover, $\dist(\clos E,\partial D)>0,$ where \[\dist(\clos E,\partial D):=\inf_{(x_1,x_2)\in \clos E \times \partial D} \dist(x_1,x_2).\]\end{assumption}

We point out that the last condition in Assumption \ref{assu:E} ensures that the set of admissible shape functions  $\FF_{\mathfrak{s}}$ is not empty, cf.\,\eqref{f_s} below.

In the rest of the paper, one tacitly supposes that Assumptions \ref{assu:stand} and \ref{assu:E} are always fulfilled without mentioning them every time. 
\subsection{The admissible set}\label{sec:as}
As pointed out in the introduction, we will work in the framework of \cite{pen, t_jde, oc_t, nt, nt2, to} (see also the references therein), where  the admissible set $\OO$ consists of a family of subdomains of $D$ that are generated by a certain class of continuous functions, called $\FF_{\mathfrak{s}}$, see \eqref{f_s} below.
For a general  mapping $g \in C(\bar D)$, we use the notation $\O_g$ to describe the following open subset of the holdall domain $D$:
\begin{equation}\label{def:og}
 \Omega_g:=\interior\{x \in D: g(x) \leq 0\} .
\end{equation}
Note that $\O_g$ is not necessarily connected and the choice of $g$ is not unique, in fact there are infinitely many functions generating the same $\O_g,$ see for instance the end of the proof of Lemma \ref{lem:g>0}. 
While $\O_g$  is an open set and may have many connected components, the
admissible shape (domain)  that we use in the definition of $\OO$ below, see \eqref{o}, is the component that contains the subdomain $E$, \ro{see Remark \ref{rem:e} below for details.}

\begin{definition}[The set of admissible shape functions]
We define
\ro{\begin{equation}\label{f_s}
\begin{aligned}
\FF_{\mathfrak{s}}&:=\{g \in C^2(\bar D):g(x) < 0\ \forall\,x\in E, \ |\nabla g(x)|+|g(x)|>0 \ \forall\,x\in D, \\&\  \ \ \ \  \ \qquad \qquad \quad g(x)>0 \ \forall\,x\in\partial D\} . \end{aligned}\end{equation}}\end{definition}

$\FF_{\mathfrak{s}}$ will be later the control set for the optimization problem \eqref{p_shh} below, as we will switch from the topology of subsets in the Euclidean space to the one in a function  space; note that \eqref{p_sh} and \eqref{p_shh} are 'equivalent' in the sense of Proposition  \ref{rem:equiv}. Instead of $C^2$ functions, we may take \ro{$C^{k,l}$ mappings in \eqref{f_s}, $k\geq 2, l \in [0,1]$} (this is possible due to Propositon \ref{dense}). Switching to \ro{$C^{k,l}$} shape functions impacts only the regularity of the boundary of $\O_g$, which is as smooth as $g$ is (see Lemma \ref{fs} or \ro{\cite[Thm.\,4.2]{dz}}). Since we want to keep the admissible set $\FF_{\mathfrak{s}}$ as large as possible and at the same time ensure the $H^2$ regularity of the state $y_g$ in \eqref{p_shh}, we choose to stick to the $C^2$-setting.

\begin{definition}[The set of admissible domains]\label{def:o}
The admissible set for the shape optimization problem \eqref{p_sh} is given by 
\begin{equation}\label{o}
\OO:=\{\text{the component of }\O_g \text{ that contains the set }E: g\in \FF_{\mathfrak{s}}\}.\end{equation} \end{definition}
 \ro{At the end of this subsection we will show that $\OO$ covers all subdomains of $D$ of class $C^2$ that contain $E$ and whose boundaries do not intersect $\partial D$.} These domains may not be simply connected, that is, topological and boundary variations are both included in our investigations.

\begin{remark}[The relevant component]\label{rem:e}
We point out that, for each $g \in \FF_{\mathfrak{s}}$, the existence of a component of $\O_g$ that contains $E$ is guaranteed by the inequality $g \ro{<} 0$ in $E$, cf.\,\eqref{f_s}. In fact, this is the only component of $\O_g$ that is relevant when solving the state equation in \eqref{p_sh}, since, in the objective, the integral term involving the state  
is taken over  $E$.

We also remark that, given a relevant component $\O \in \OO$, there is an infinite number of mappings  $g \in \FF_{\mathfrak{s}}$ so that $\O_g=\O,$ see Lemma \ref{lem:g>0} below.
\end{remark}

\begin{remark}[Comments regarding the definition of $\FF_{\mathfrak{s}}$]\label{rem:f_s}
Let us underline  that all the conditions appearing in the definition of $\FF_{\mathfrak{s}}$ are meaningful, as  they   guarantee that the admissible shapes have the properties stated  in Lemma \ref{fs} below.
 These are mostly used for \eqref{h_eq}, in the proof of Proposition \ref{prop:cor} and in subsection \ref{sec:s_e}.  Lemma \ref{fs}  will also play a central role when establishing  the limit adjoint equation in \cite{p3}.
 \\The inequality $g\ro{<} 0$ in $E$ implies  that all admissible shapes contain the subdomain $E$, as already mentioned. 
\\Since $g \in C^2(\bar D)$, the requirements $|\nabla g|+|g|>0 \text{ in }D$ and $ g>0 \text{ on }\partial D$ ensure that the set $\{x \in D:g(x)=0\}$ is a finite union of disjoint closed $C^2$ curves, without self intersections, not intersecting $\partial D$ \cite[Prop.\,2]{pen}. \ro{The fact that we  work in two dimensions only is  due to the viability of this assertion. If the three-dimensional case is considered, we need to include  more conditions in $\FF_{\mathfrak{s}}$ \cite{impl, impl3}.}
\end{remark}

\begin{lemma}[Properties of admissible shapes and $\O_g$]\label{fs}
Let $g \in \FF_{\mathfrak{s}}$ and denote by $\O \in \OO$ the relevant component of $\O_g,$ that is, the component that contains $E$. Then,
\begin{enumerate}
\item \label{class} $ \O \text{ is a domain of class }C^2$;
\item \label{class2} $\partial \O_g=\{x \in D: g(x) = 0\}$ and $\O_g=\{x \in D: g(x) < 0\};$    
\item \label{ls} $\mu\{x \in D: g(x) = 0\}=0$.
\end{enumerate}
\end{lemma}
\begin{proof} 
We observe that $\partial \O_g \subset \{x \in D: g(x) = 0\}$. As $g \in C^2(\bar D)$ and since it satisfies the gradient condition and the condition on the boundary of $D$ in \eqref{f_s}, one deduces from \cite[Prop.\,2]{pen}   that $\{x \in D: g(x) = 0\}$ is a finite union of disjoint closed $C^2$ curves, without self intersections, disjoint from $\partial D$. This implies the first claim.

\ro{Further, from \cite[Thm.\,4.2]{dz} it follows, by a contradiction argument, that $\interior \{x \in D: g(x)\leq 0\}=\{x \in D: g(x) < 0\}$ and $\partial \{x \in D: g(x) < 0\}=\{x \in D: g(x) = 0\}$. This gives us the second statement.}

The last assertion is  a direct consequence of the first two. It can also be deduced directly from  \cite[Lemma A.4]{ks} and the gradient condition in the definition of $\FF_{\mathfrak{s}}$.
\end{proof}                                                                                                                               




\ro{We end this subsection with an essential observation.
\begin{proposition}\label{prop:c2}
It holds 
\[\OO=\{\O \subset D:\O \text{ is a domain of class $C^{2}$  that contains }E, \ \partial \O \cap \partial D \neq \emptyset\}.\]
\end{proposition}
\begin{proof}While the inclusion $\subset$ is due to Lemma \ref{fs}, the opposite one can be shown by the exact same 
  arguments as in the proof of \cite[Thm.\,4.1, sec. 4, p.75]{dz} with one difference: one replaces the distance functions defined therein by appropriate  $C^\infty$ functions. Let $\O \subset D$ so that $\O \text{ is a domain of class $C^{2}$  that contains }E, \ \partial \O \cap \partial D \neq \emptyset$. To show that there exists $g \in \FF_\mathfrak{s}$ so that $\ogg=\O$, one defines
 $\rho,\rho_c \in C^\infty(\bar D)$  as follows
\begin{equation}
\rho(y):=\left\{
\begin{aligned}
\in (0,1] &\quad \text{for $y \in  \O^c \setminus \bar W$},
\\0 &\quad \text{for $y \in \bar W \cup \O=\bar W \cup \bar \O$},
\end{aligned}\right.
\end{equation} 
\begin{equation}
\rho_c(y):=\left\{
\begin{aligned}
\in (0,1] &\quad \text{for $y \in  \O \setminus \bar W$},
\\0 &\quad \text{for $y \in \bar W \cup \O^c=\bar W \cup \bar \O^c$},
\end{aligned}\right.
\end{equation} 
where  $\O^c$ stands for $\bar D \setminus \O$ and $W$ is an appropriate neighbourhood of $\partial \O,$ see the proof of \cite[Thm.\,4.1, sec. 4, p.75]{dz} for details.
Then, by arguing exactly in the same way as in the proof of \cite[Thm.\,4.1, sec. 4, p.75]{dz}, the desired result follows. \end{proof}}

\subsection{Reformulation of  \eqref{p_sh} as an optimal control problem in a function space}\label{sec:ref}
In this subsection we rewrite \eqref{p_sh} as an optimal control problem where the admissible set no longer consists of subsets of $D$, but of functions. In other words, we change the topology, so that \eqref{p_sh}  becomes  amenable to  optimal control  methods. This is possible in light of  the definition of the set of admissible shapes, see \eqref{o}. The reformulation of \eqref{p_sh} reads
\begin{equation}\tag{$P$}\label{p_shh}
 \left.
 \begin{aligned}
  \min_{g \in \FF_{\mathfrak{s}}} \quad & \int_E (y_{g}(x)-y_d(x))^2 \; dx+\alpha\,\int_{D}1-H(g)\,dx, 
\\     \text{s.t.} \quad & 
  \begin{aligned}[t]
    -\laplace y_{g} + \beta(y_{g})&=f \quad \text{in }\O_g,
   \\y_{g}&=0  \quad \text{on } \partial \O_g, \end{aligned} \end{aligned}
 \quad \right\}
\end{equation}where $H:\R \to\{0,1\}$ is the Heaviside function \begin{equation}\label{h}
H(v):=\left\{\begin{aligned}0,&\quad \text{if }v\leq 0,
\\1&\quad \text{if }v>0.
\end{aligned}\right.\end{equation}
Regarding the second term in the objective, we note that 
\begin{equation}\label{h_eq}\int_{D}1-H(g)\,dx=\int_{\ogg}\,dx \quad \forall\,g \in \FF_{\mathfrak{s}},\end{equation}in view of Lemma \ref{fs}.

\begin{lemma}\label{lem:g>0}
For each $\O \in \OO$ there exists  an infinite number of mappings $g \in \FF_{\mathfrak{s}}$ so that $\O_g=\O$.\end{lemma}
\begin{proof}
Since $\O \in \OO$, there exists 
$\widehat g \in \FF_{\mathfrak{s}}$ so that $\O$ is the component of $\O_{\widehat g}$ that contains $E$, cf.\,\eqref{o}.
From \cite[Prop.\,2]{pen}, see also the proof of Lemma \ref{fs}, we know that $\O_{\widehat g}$ has a finite number of components, and we define
\ro{\[\MM:=\cup_{\omega \neq \O, \\ \o \text{ is a component of }\O_{\widehat g}} \clos \o.\]
Note that $\dist(\clos{ \O},\clos \o)>0$ for each component $\o$ of $\O_{\widehat g}$ which is different from $\O$ (the boundary of $\O_{\widehat g}$ consists of disjoint closed curves  \cite[Prop.\,2]{pen}, see also the proof of Lemma \ref{fs}).
Now, 
define the $C^\infty$ function
\begin{equation}
h:=\left\{
\begin{aligned}
0 &\quad \text{for $x \in \clos \O$},
\\\in [0,2]&\quad \text{for $x \in \clos D \setminus (\clos \O \cup \MM)$},
\\2 &\quad \text{for $x \in \MM$}.
\end{aligned}\right.
\end{equation}
and
$$ g:= \widehat g-(\min_{x \in \clos D} \widehat g) h.$$ Note that  $\min_{x \in \clos D} \widehat g < 0,$ otherwise $\widehat g\geq 0$ in $\clos D$, which contradicts $\widehat g \in \FF_{\mathfrak{s}}$.}
We observe that 
\[g= \widehat g \ \text{ in }\clos{ \O}, \quad g>0 \ \text{ in }\bar D \setminus \clos{ \O},\]
since $\widehat g\ro{> }0$ outside $\clos \O_{\widehat g}$ (Lemma \ref{fs}). As $\widehat g \in \FF_{\mathfrak{s}}$, we can conclude from the above that $g \in \FF_{\mathfrak{s}}$ with $\O_{g}=\O.$ Finally, we see that one can define an infinite number of mappings with this property, say \ro{$ g_n:= n\,g, n \in \N$.}
\end{proof}

The shape optimization problem \eqref{p_sh} with the admissible set $\OO$ from Definition  \ref{def:o} is equivalent to \eqref{p_shh} in the following sense. 
\begin{proposition}\label{rem:equiv}
Let  $\O^\star \in \OO$ be an optimal shape of \eqref{p_sh}. Then, each of the  functions  $g^\star \in \FF_{\mathfrak{s}}$ that satisfy $\O_{g^\star}=\O^\star$ is a global minimizer of \eqref{p_shh}. Conversely, if $g^\star \in \FF_{\mathfrak{s}}$ minimizes \eqref{p_shh}, then the component of $\O_{g^\star}$ that contains $E$ is an optimal shape for \eqref{p_sh}.
\end{proposition}
\begin{proof}
Let $ \O^\star \in \OO$ be an optimal shape for \eqref{p_sh} with $\OO$ as in \eqref{o}, i.e.,
\[\int_E (y_{ \O^\star}(x)-y_d(x))^2 \; dx+\alpha\,\int_{ \O^\ast}\,dx
\leq \int_E (y_{ \O}(x)-y_d(x))^2 \; dx+\alpha\,\int_{\O}\,dx \quad \forall\,\O \in \OO.
\]
Now, let $g^\star \in \FF_{\mathfrak{s}}$ with $\O_{g^\star}=\O^\star$ be fixed (note that, according to Lemma \ref{lem:g>0}, there are infinitely many mappings with this property).  
Then, 
$y_{\O^\star}=y_{ g^\star}$ and 
\[\int_E (y_{g^\star}(x)-y_d(x))^2 \; dx+\alpha\,\int_{\O_{ g^\star}}\,dx
\leq \int_E (y_{ \O}(x)-y_d(x))^2 \; dx+\alpha\,\int_{\O}\,dx \quad \forall\,\O \in \OO.
\]
Let $g \in \FF_{\mathfrak{s}}$ be arbitrary and fixed and denote by $\O$ the component of $\O_{ g}$ that contains $E$. Testing with this particular $\O$ in the above inequality yields 
\begin{align*}\int_E (y_{g^\star}(x)-y_d(x))^2 \; dx+\alpha\,\int_{\O_{g^\star}}\,dx
&\leq \int_E (y_{ g}(x)-y_d(x))^2 \; dx+\alpha\,\int_{\O}\,dx \\&\leq \int_E (y_{ g}(x)-y_d(x))^2 \; dx+\alpha\,\int_{\O_g}\,dx ,\end{align*}
where we used that $y_g=y_\O$ in $E$. 
Since $g^\star \in \FF_{\mathfrak{s}}$ and $g \in \FF_{\mathfrak{s}}$ was arbitrary, this proves the first statement, see \eqref{h_eq}.

To show the converse assertion, assume that $g^\star \in \FF_{\mathfrak{s}}$ satisfies
\begin{equation}\label{inn}
\int_E (y_{g^\star}(x)-y_d(x))^2 \; dx+\alpha\,\int_{\O_{g^\star}}\,dx
\leq \int_E (y_{ g}(x)-y_d(x))^2 \; dx+\alpha\,\int_{\O_g}\,dx \quad \forall\,g \in \FF_{\mathfrak{s}}.\end{equation}
We denote by $\o_{g^\star}$  the component of $\O_{g^\star}$ that contains $E$. This implies 
\[y_{\o_{g^\star}}=y_{ g^\star} \quad \text{in }E.\]
Let $\O \in \OO$ be arbitrary but fixed. Again, by Lemma \ref{lem:g>0}, we can define $\widetilde g \in \FF_{\mathfrak{s}}$ so that $\O_{\widetilde g}=\O.$ Then,
$y_{\O}=y_{\widetilde g}$. In view of \eqref{inn}, where we test with $\widetilde g$, we have
\begin{align*}
\int_E (y_{\o_{g^\star}}(x)-y_d(x))^2 \; dx+\alpha\,\int_{\o_{g^\star}}\,dx
&\leq 
\int_E (y_{g^\star}(x)-y_d(x))^2 \; dx+\alpha\,\int_{\O_{g^\star}}\,dx
\\&\leq \int_E (y_{ \O}(x)-y_d(x))^2 \; dx+\alpha\,\int_{\O}\,dx.
\end{align*}
Since $\o_{g^\star} \in \OO$ and $\O \in \OO$ was arbitrary, this proves the second assertion.
\end{proof}

From now on, all our findings concern the optimal control problem formulated as  \eqref{p_shh}. 
We focus not only on global minimizers, but on the much larger class of local optima in the $\ld$-sense, which are defined as follows.
\begin{definition}\label{def_sh}
We say that $\bar g_{\mathfrak{s}} \in \FF_{\mathfrak{s}}$ is locally optimal for \eqref{p_shh} in the $\ld$ sense, if there exists $r>0$ such that 
\begin{equation}\label{loc_opt_shh}
\JJ(\bar g_{\mathfrak{s}}) \leq \JJ(g) \quad \ff g \in \FF_{\mathfrak{s}} \text{ with }\|g-\bar g_{\mathfrak{s}}\|_{\ld}\leq  r,
\end{equation}where $$\JJ(g):=\int_E (\SS( g)(x)-y_d(x))^2 \; dx+\alpha\,\int_{D} (1-H(g))(x)\,dx$$ is the reduced cost functional associated to the control problem \eqref{p_shh}  and  \[\SS:g \in \FF_{\mathfrak{s}} \mapsto y_g \in H_0^1(\O_g) \cap H^2(\O_g)\]
 denotes the control-to-state map. This will be introduced below, see Definition \ref{S}.
\end{definition}
\begin{remark}
Clearly, in the study of local optima in the $\ld$ sense of \eqref{p_shh} we cover global optima, and thus, the associated (global) optimal shapes of \eqref{p_sh}, cf.\, Proposition \ref{rem:equiv}.

Note that the set of local optima in the $H^1(D)$-sense is larger than the set of local optima in the $\ld$-sense. However,  we choose to work with the concept of local optima in the $\ld$-sense, in view of the essential density property from Proposition \ref{dense} below. This does not hold w.r.t.\,the $H^1(D)$ norm because of  the boundary condition on $\partial D$ in \eqref{f_s}.  Proposition \ref{dense} is however essential in the proof of Theorem \ref{thm:cor}, as a closer inspection shows; it ensures that $\{\widehat g_\eps\}\subset \FF_{\mathfrak{s}}$ is contained in the ball of local optimality of $\bar g_{\mathfrak{s}}$. If this ball of local optimality is taken w.r.t.\,the $H^1(D)-$norm, then it is necessary that Proposition \ref{dense} is also true w.r.t.\,the $H^1(D)-$norm, which is not the case, as explained above.
\end{remark}

\begin{remark}[Relation between the notion of local optimality in Definition \ref{def_sh} and 'local' optimal shape]
If two admissible shapes are 'close' to each other, this does not mean that the same is true for their parametrizations, since, given $g\in \FF_{\mathfrak{s}}$, there is an infinity of functions in $\FF_{\mathfrak{s}}$  generating $\O_g \in \OO$ (see Lemma \ref{lem:g>0}). In fact, for each $ g \in \FF_{\mathfrak{s}}$, there is a sequence $\{g_n\}\subset \FF_{\mathfrak{s}}$ so that $\O_{ g}=\O_{g_n}, \ n \in \N,$ and at the same time $\|g_n-g\|_{\ld} \to \infty$ as $n \to \infty$ (cf.\,the end of the proof of Lemma \ref{lem:g>0}).
\\
On the other hand, if $\|g_n-g\|_{\ld} \to 0$, $g_n, g \in \FF_{\mathfrak{s}}$, we have
\[\mu\{x \in D: g>0 \text{ and }g_n \leq 0\} \to 0 \quad \text{as }n \to \infty,\]
\[\mu\{x \in D: g<0 \text{ and }g_n \geq 0\} \to 0 \quad \text{as }n \to \infty,\] see Lemma \ref{lem:app}. 
In light of Lemma \ref{fs}, this accounts to
\[\mu(\O_{g_n}\setminus \O_{ g})\to 0 \quad \text{as }n \to \infty,\]
\[\mu( \O_{ g}\setminus \O_{ g_n})\to 0 \quad \text{as }n \to \infty.\]
It is an open question how to define the notion of 'local optimal shape' so that it makes sense in connection with the $L^2$ local optimality of the associated parametrizations; this may involve the Hausdorff-Pompeiu distance or other notions of distances  between two sets. The above convergences may be a good starting point for further investigations on this particular topic.
At this stage, it is however clear that our method covers all global optimal shapes (Proposition \ref{rem:equiv} and Corollary \ref{cor:os}). \end{remark}

\section{Density of the set of admissible shape functions in $L^p(D)$, $p\in[1,\infty)$}\label{sec:3'}
As already mentioned in the introduction,   we cannot tackle \eqref{p_shh} in a  direct manner, so that we will consider an approximation scheme, cf.\,sections \ref{sec:3} and \ref{cor} below. 
One of the main challenges that arise in the investigation of \eqref{p_shh} is the structure of the set $\FF_{\mathfrak{s}}$. It is a non-convex  cone, while the governing PDE is non-smooth and the control appears as the parametrization of the unknown domain. 
The admissible set $\FF_{\mathfrak{s}}$ has however a density property, which allows us to examine  the approximating optimal control problem \eqref{p10} below on a convex subset of a Hilbert space,  see \eqref{f_tilde}. 

The focus of this subsection is to prove the aforementioned density property, which is stated in the following (in a slightly more general form):
\begin{proposition}\label{dense}
Let \[\widetilde \FF:=\{g \in L^2(D):g \leq 0 \text{ a.e.\,in }E\}.\]
For each $g \in \widetilde \FF$,  there exists a sequence $\{g_m \} \subset  \FF_{\mathfrak{s}} \cap C^\infty(\bar D)$ so that $$g_m \to g \quad \text{in }L^2(D) \quad \text{ as }m \to \infty.$$
\end{proposition}
\begin{remark}\label{rem:dense}
The assertion in Proposition \ref{dense} remains true if $\ld$ is replaced by $L^p(D)$, $p\in[1,\infty)$, see \eqref{c_m} and \eqref{lp} below. 

It has different applications for all sort of problems where the solution operator is defined on $\FF_{\mathfrak{s}}$ only, i.e., where conditions as those in Lemma \ref{fs} need to be fulfilled by the admissible shapes, see e.g.\,\cite{t_jde} and the references therein.
\end{remark}

The proof of Proposition \ref{dense} will be conducted in three steps, as follows.
We introduce the sets
\[\FF_c:=\{g \in  C^\infty( \bar D): g(x) \leq 0 \ \forall\,x\in E\},\]
\begin{equation}
\begin{aligned}
\FF_c^+:=\{g \in C^\infty(\bar D): g(x) \leq 0\ \forall\,x\in E, \ g(x)>0 \ \forall\,x\in\partial D\} \end{aligned}\end{equation}
and prove Proposition \ref{dense} by showing  the set of densities 
\begin{equation}
\begin{aligned}
\FF_c \dense \widetilde \FF \quad (\text{Lemma }\ref{lem_r} \text{ below}),
\\\FF_c^+ \dense\FF_c \quad (\text{Lemma }\ref{lem_1} \text{ below}),
\\\FF_{\mathfrak{s}} \cap C^\infty(\bar D) \dense\FF_c^+ \quad (\text{Lemma }\ref{lem_2} \text{ below})   \end{aligned}\end{equation}
w.r.t.\,the $\ld$ norm.

\begin{lemma}\label{lem_r}
For each $g \in \widetilde \FF$,  there exists a sequence $\{g_m \} \subset \FF_c$
so that $$g_m \to g \quad \text{in }\ld \quad \text{ as }m \to \infty.$$
\end{lemma}
\begin{proof}
Let $g \in \ld $ with $g \leq 0$ a.e.\,in $E$ be arbitrary, but fixed and denote by $\chi_D g:\R^2 \to \R$ the extension of $g$ by zero outside $D$. 
For every $m \in \N$ we define the $\ld$ function $\widehat g_m:\R^2 \to \R$ as 
\begin{equation}
 \widehat g_m:=\left\{
\begin{aligned}
\min\{g,0\} &\quad \text{on } E_m,
\\\chi_D g &\quad \text{otherwise},
\end{aligned}\right.
\end{equation}where 
$$E_m:=\{v \in \R^2: d(v, E) \leq 1/m\},$$and  $d$ is the distance induced by the $l_\infty$ norm; note that $E_m \subset D$ for $m$ large enough, by Assumption \ref{assu:E}. Since $g\leq 0$ a.e.\,in $E$, we have 
\begin{equation}\label{c_m}
\|\widehat g_m-g\|_{\ld}^2=\|\min\{g,0\} -g\|_{L^2(E_m \setminus E)}^2=\int_{E_m \setminus E}(-\max\{g,0\})^2  \to 0 \quad \text{as }m \to \infty, \end{equation}since $(\max\{g,0\})^2 \in L^1(D)$ and $\mu (E_m \setminus E) \to 0 \text{ as }m \to \infty,$ cf.\,e.g.\,\cite[Lemma A.1.17]{alt}.
Now we define the $C^{\infty}$ function $g_m:\R^2 \to \mathbb{R}$ as
\[g_m(v):=\int_{\R^2} \widehat g_m(v- \frac{1}{m} s)\psi(s)\,ds,\]
where $\psi \in C_c^\infty(\mathbb{R}^2),\ \psi \geq 0, \ \supp \psi \subset [-1,1] ^2$ and $\int_{\R^2} \psi(s)\, ds=1$. 
Then, by the definition of $E_m$ and $\widehat g_m$, it holds $$g_m(v)=\int_{[-1,1]^2} \widehat g_m(v- \frac{1}{m} s)\psi(s)\,ds \leq 0 \quad \forall\, v \in E,$$ i.e., $g_m  \in \FF_c$.
In view of \eqref{c_m}, the proof is now complete.
\end{proof}

\begin{lemma}\label{lem_1}
For each $g \in  \FF_{c}, $ there exists a sequence $\{g_m \} \subset \FF_c^+,$
so that $$g_m \to g \quad \text{in }\ld \quad \text{ as }m \to \infty.$$
\end{lemma}
\begin{proof}We recall that
\[\FF_c:=\{g \in  C^\infty( \bar D): g(x) \leq 0 \ \forall\,x\in E\}\]and 
\begin{equation}
\begin{aligned}
\FF_c^+:=\{g \in C^\infty(\bar D): g(x) \leq 0\ \forall\,x\in E, \ g(x)>0 \ \forall\,x\in \partial D\} . \end{aligned}\end{equation}

Let $g \in  \FF_{c}$ be arbitrary but fixed and $m \in \N$. 
We define $h_m \in C^\infty(\bar D)$ as
\begin{equation}\label{hm}
h_m(x):=\left\{
\begin{aligned}
-2\min\{\min_{x \in \partial D} g,0\} +1/m  &\quad \text{for $x \in \partial D$},
\\\in [0,-2\min\{\min_{x \in \partial D} g,0\}  +1/m]&\quad \text{for $x \in  D$ where }d(\partial D,x) \in (0,1/m),
\\0 &\quad \text{for $x \in  D$ where }d(\partial D,x) \geq 1/m.
\end{aligned}\right.
\end{equation}
Then, the mapping \begin{equation}\label{gm}
g_m:=g+h_m \in C^\infty(\bar D) \end{equation} 
satisfies
\begin{equation}\label{lp}
\|g_m-g\|_{\ld}=  \|h_m\|_{L^2(D_m)} \to 0 \quad \text{as }m \to \infty, \end{equation}
where we abbreviate  \[D_m:=\{ x \in  D: d(\partial D,x) \in(0, 1/m)\}.\]
Let us now check if $g_m$ fulfils the inequalities appearing in the definition of $\FF_c^+$.
\\\textit{(i) $g_m \leq 0$ in $E$}:  By Assumption \ref{assu:E},  we can choose $m$ large enough so that $1/m <d(\bar E, \partial D)$. This implies that 
$h_m=0 \text{ in }\bar E$,
which gives in turn  $g_m=g\leq 0$ in $\bar E$.
\\\textit{(ii) $g_m>0$ on $\partial D:$} 
By \eqref{gm}, we have for $x \in \partial D$:
\[ g_m(x)\geq \min\{\min_{x \in \partial D} g,0\} -2\min\{\min_{x \in \partial D} g,0\} +1/m=-\min\{\min_{x \in \partial D} g,0\} +1/m\geq 1/m,\]
where for the first estimate we used the definition of $h_m$, i.e., \eqref{hm}. This allows us to conclude the desired assertion.
\end{proof}


\begin{lemma}\label{lem_2}
For each $g \in  \FF_{c}^+, $ there exists a sequence $\{g_m \} \subset  \FF_{\mathfrak{s}} \cap C^\infty(\bar D),$
so that $$g_m \to g \quad \text{in }C(\bar D) \quad \text{ as }m \to \infty.$$
\end{lemma}
\begin{proof} We recall the definitions of $\FF_c^+$ and $  \FF_{\mathfrak{s}}:$ \begin{equation}
\begin{aligned}
\FF_c^+:=\{g \in C^\infty(\bar D): g(x) \leq 0\ \forall\,x\in E, \ g(x)>0 \ \forall\,x\in \partial D\} , \end{aligned}\end{equation}
\begin{equation}
\begin{aligned}
 \FF_{\mathfrak{s}}&:=\{g \in C^2(\bar D):g(x) \ro{<} 0\ \forall\,x\in E, \ |\nabla g(x)|+|g(x)|>0 \ \forall\,x\in D, \ g(x)>0  \ \forall\,x\in \partial D\} . \end{aligned}\end{equation}
 For the  definition of $\O_g$, see \eqref{def:og}.
 Let $g \in  \FF_{c}^+$ be arbitrary but fixed. 
According to Sard's theorem {\cite[Thm.\,7.2]{sard}}, the set 
\begin{equation}\label{gm0}
\{g(x) : \nabla g(x)=0\}\subset \R \ \text{ has measure zero.} \end{equation}Since $g$ is continuous and $g>0$ on $\partial D$, it holds $\min_{x \in \partial D}g(x)>0$. In light of \eqref{gm0}, there exists $\delta_1 \in (0,\min_{x \in \partial D}g(x))$ so that
\[\nabla g(x)\neq 0 \quad \forall\,x \in g^{-1}(\delta_1).\]
Note that $g^{-1}(\delta_1)$ may be empty, in which case the above assertion is still valid.
Sard's theorem, see \eqref{gm0}, further ensures the existence of a sequence $\{\delta_m\} $ satisfying \[\delta_m \in (0,\min\{\delta_{m-1},\frac{1}{m}\}), \quad m\in \N\]and 
\begin{equation}\label{ngm}
\nabla g(x)\neq 0 \quad \forall\,x \in g^{-1}(\delta_m) \quad \forall\,m \in \N.\end{equation}
We notice that the above constructed sequence satisfies\begin{equation}\label{delta}
\min_{x \in \partial D}g(x)>\delta_1>\delta_2>...>\delta_m>0, \quad \delta_m \to 0 \quad \text{as }m \to \infty.\end{equation}
Further, we define $f_m \in C^\infty(\bar D)$ as \begin{equation}\label{fm'}
f_m:=\left\{
\begin{aligned}
\in (0,1] &\quad \text{for $x \in  \bar D \setminus \overline {\O_{g-\delta_m}}$},
\\0 &\quad \text{for $x \in \overline {\O_{g-\delta_m}}$},
\end{aligned}\right.
\end{equation}where we recall that, according to \eqref{def:og}, the set $\O_{g-\delta_m}$ is defined as
\begin{equation}\label{ogdm}
\O_{g-\delta_m}=\interior\{x \in D:g(x)-\delta_m \leq 0\}.\end{equation}
Note that  functions $f_m$ as in \eqref{fm'} exist, since we can choose a compact subset of $\R^2$, disjoint with $\bar D$, so that $f_m=1$ there. 
Then, the mapping \begin{equation}\label{gm'}
g_m:=g-\delta_m +\frac{1}{m}f_m\in C^\infty(\bar D) \end{equation} 
satisfies
\begin{equation}
\|g_m-g\|_{C(\bar D)}  \to 0 \quad \text{as }m \to \infty,\end{equation}see \eqref{delta}.
Let us now check if $g_m$ fulfils the inequalities appearing in the definition of $ \FF_{\mathfrak{s}}$.
\\\textit{(i) $g_m \ro{<} 0$ in $E$}:  Since $g\in \FF_c^+,$ it holds $g\leq 0$ in $\bar E$, which implies that 
$$\bar E \subset \{x \in D:g(x)-\delta_m \ro{<} 0\},$$ \ro{from which }we infer $$\bar E \subset \overline {\O_{g-\delta_m}}.$$ This yields $f_m=0$ in $\bar E$, see \eqref{fm'}. Thus, by \eqref{gm'}, $g_m=g -\delta_m \leq \ro{-\delta_m <0}$ in $\bar E$.
\\\textit{(ii) $g_m>0$ on $\partial D:$} 
From \eqref{delta} one deduces that for $x \in \partial D$ it holds 
\[ g(x)-\delta_m >g(x)-\min_{x \in \partial D} g \geq 0.\]
From \eqref{gm'} and \eqref{fm'} we now conclude that $g_m>0$ on $\partial D.$
\\\textit{(iii)}\ro{ $|\nabla g_m(x)|+|g_m(x)|>0 \ \forall\,x\in  D$: By \eqref{gm'} and \eqref{fm'} one knows that
\begin{equation}\label{1}g_m(x)>0 \ \forall\,x\in \bar D \setminus \overline {\O_{g-\delta_m}} .\end{equation}Thus, we only need to show that 
\[|\nabla g_m(x)|+|g_m(x)|>0 \ \forall\,x\in  \overline{\O_{g-\delta_m}}.\]Note that, due to the regularity and the definition of $f_m$, see \eqref{fm'}, its gradient vanishes  on $\overline{\O_{g-\delta_m}}$. This means that, by \eqref{gm'}, it holds 
\[g_m=g-\delta_m \text{ in }\overline{\O_{g-\delta_m}}, \quad \nabla g_m=\nabla g \quad \text{ in }\overline{\O_{g-\delta_m}}.\]
Now, thanks to \eqref{ngm}, $g_m(x)=0$ implies $\nabla g_m(x)\neq 0,$ if $x\in \overline{\O_{g-\delta_m}}.$ This completes the proof.}
\end{proof}

\section{The approximating optimal control problem}\label{sec:3}
In this section,  $\bar g_{\mathfrak{s}} \in \FF_{\mathfrak{s}}$ is a local optimum of \eqref{p_shh} in the $\ld$-sense. 
Inspired by a classical adapted penalization scheme  \cite{barbu84} in combination with the fixed domain methodology from \cite{nt}, our approximating minimization problem \eqref{p10} preserves the non-smoothness. We point out  that the non-convex set $\FF_{\mathfrak{s}}$ from \eqref{p_shh} is now replaced by a convex subset of a Hilbert space. As shown by one of the main results in the next section (Theorem \ref{thm:cor}), this replacement is reasonable thanks to the findings from the previous section (Proposition \ref{dense}).

Let $\eps>0$ be fixed. We consider the following approximating control problem\begin{equation}\tag{$P_\eps$}\label{p10}
 \left.
 \begin{aligned}
  \min_{g \in   \FF} \quad & \int_E (y(x)-y_d(x))^2 \; dx+\alpha\,\int_{D} (1-H_\eps(g))(x)\,dx+\frac{1}{2}\,\|g-\bar g_{\mathfrak{s}}\|_{\WW}^2  \\   \text{s.t.} \quad & 
  \begin{aligned}[t]
   -\laplace y + \beta(y)+\frac{1}{\eps}H_\eps(g)y&=f +\eps g\quad \text{in }D,
   \\y&=0\quad \text{ on } \partial D, \end{aligned} 
 \end{aligned}
 \quad \right\}
\end{equation}where $\WW$ is the Hilbert space $\ld \cap \hde$, $s\in(1,2]$, endowed with the norm 
\[\|\cdot\|_{\WW}^2:=\|\cdot\|_{\ld}^2+\|\cdot\|_{H^s(D\setminus \bar E)}^2,\]
and 
\begin{equation}\label{f_tilde}
 \FF:=\{g \in \WW: g \leq 0\text{ a.e.\,in }E\}.
\end{equation}

\begin{definition}The non-linearity $H_\eps:\R \to [0,1]$ is defined as follows 
 \begin{equation}\label{reg_h}
H_\eps(v):=\left\{\begin{aligned}0,&\quad \text{if }v\leq 0,
\\\frac{v^2(3\eps-2v)}{\eps^3},&\quad \text{if }v\in (0,\eps),
\\1,&\quad \text{if }v\geq \eps.
\end{aligned}\right.\end{equation}
\end{definition}


\begin{remark}The mapping $H_\eps$ introduced in \eqref{reg_h}  is Lipschitz continuous and continuously differentiable.
It is obtained as a regularization of the Heaviside function from \eqref{h}
and it is one possible choice that has the aforementioned properties. Note that Heaviside functions and their regularizations play an essential role  in the context of shape optimization via fixed domain approaches, see for instance \cite{nt, nt2} and the references therein. 
\end{remark}
\begin{remark}[Additional term in the state equation]
We notice that the state equation in \eqref{p10} is an approximating extension of the state equation in \eqref{p_shh} from $\O_g$ to $D$ \cite{nt}. We point out the presence of the additional term $+\eps g$ on the right hand side, which is essential for the proof of our main strong stationarity result in {\cite{p2}}, see \cite[Rem.\,3.16]{p2} for details. In all the other upcoming investigations, this can be dropped.
\end{remark}

\begin{remark}[Objective in \eqref{p10}]
The presence of the $H^s$ term in the objective of \eqref{p10} ensures the existence of optimal solutions. We underline that it is necessary to consider the $H^s(D\setminus \bar E)$ norm (not the full $H^s(D)$ norm), in order to conclude complete (i.e., strong stationary) optimality conditions for local optima of  \eqref{p10}, see {\cite{p2}}. If $\WW=H^s(D)$, certain sign conditions for the adjoint state are not  available, and the optimality system for \eqref{p10} reduces to a weaker one, namely the limit optimality system obtained via smoothening methods.
\\Moreover, it is essential that $s>1$ in order to obtain the convergence of the states in Theorem \ref{thm:cor}, see also Lemma \ref{lem:cp}.  This condition ensures that the embedding $\hde \embed L^\infty(D \setminus \bar E)$ is true \cite[p.88]{kaballo}. It is also necessary when deriving the final optimality system for the optimal shape of \eqref{p_sh} in \cite{p3}. On the other side, the upper bound $s\leq 2$ is only needed to guarantee the inclusion $\FF_\mathfrak{s} \subset \WW$, so that the last term   in the objective of \eqref{p10} is well-defined, since $\bar g_\mathfrak{s} \in \FF_\mathfrak{s}.$ \end{remark}

\subsection{Solvability of the state equations in \eqref{p10} and \eqref{p_shh}}\label{sec:solv}
We start this subsection with a result on the unique solvability of the state equation in the approximating control problem. For convenience, we recall  here the governing PDE in \eqref{p10}:
\begin{equation}\label{eq}
  \begin{aligned}
   -\laplace y + \beta(y)+\frac{1}{\eps}H_\eps(g)y&=f +\eps g\quad \text{in }D,
   \\y&=0\quad \text{on } \partial D.
 \end{aligned}
\end{equation}

\begin{lemma}\label{lem:S}
For any right-hand side $g \in L^2(D)$, the equation \eqref{eq} admits a unique solution $ y \in H^1_0(D) \cap H^2(D)$. 
This satisfies $$\|y\|_{H^1_0(D) \cap C(\bar D)} \leq c_1+c_2\,\|g\|_{\ld},$$where $c_1,c_2>0$ are independent of $\eps$, $\beta$, and  $g$.
The control-to-state mapping $S_\eps:L^2(D) \to H^1_0(D) \cap H^2(D)$ is  Lipschitz continuous on bounded sets, i.e.,
for every  $M > 0$, 
  there exists $L_{M,\eps} > 0$ so that
  \begin{equation}\label{slip}
   \|S_\eps(g_1) - S_\eps(g_2)\|_{H^1_0(D) \cap H^2(D)} \leq L_{M,\eps} \, \|g_1 - g_2\|_{L^2(D)} \quad \forall\, g_1,g_2 \in \clos{B_{\ld}(0,M)}.
  \end{equation}

\end{lemma}
\begin{proof}According to \cite[Thm.\,4.8]{troe},  the equation \eqref{eq} admits a unique solution $ y \in H^1_0(D) \cap C(\bar D)$ with
$$\|y\|_{ \hdd \cap \li} \leq c\,\|f+\eps g-\beta(0)\|_{\ld},$$ where $c>0$ is independent of $\eps$, $\beta$, and of $g$. Then, by \cite[Lem.\,9.17]{gt}, the $H^2(D)$ regularity  follows.

To show the desired 
 Lipschitz continuity of the control-to-state map, let $g_1,g_2 \in L^2(D)$ be arbitrary but fixed and abbreviate $y_i:=S_\eps(g_i), i=1,2.$ As a consequence of the monotonicity of $\beta$ and the Lipschitz continuity and non-negativity of $H_\eps,$ cf.\,\eqref{reg_h}, one obtains
 \begin{align}
 \|y_1-y_2\|_{H^1_0(D)}^2 &\leq \eps \|g_1-g_2\|_{\ld} \|y_1-y_2\|_{\ld}
 \\&\quad +\frac{1}{\eps}\|H_\eps(g_1)-H_\eps(g_2)\|_{\ld}\|y_2\|_{C(\bar D)}\|y_1-y_2\|_{\ld}
 \\&\leq \|g_1-g_2\|_{\ld} \|y_1-y_2\|_{\ld}(\eps+\frac{L_{H_\eps}}{\eps}c\,\|f+\eps g_2-\beta(0)\|_{\ld}),
 \end{align}where $c>0$ is independent of $\eps$, $\beta$, and of $g_2$.
 This means that $S_\eps:\ld \to H^1_0(D)$ is Lipschitz continuous on bounded sets. In view of the $C^{1,1}$ regularity of $D$, the same can be concluded for $S_\eps:\ld \to H^1_0(D) \cap H^2(D).$
\end{proof}

Next we aim at highlighting the connection between \eqref{eq} and the state equation in \eqref{p_shh}. 

\begin{proposition}[Solvability of the state equation in \eqref{p_shh}]\label{prop:cor}
Let $g\in C(\bar D)$ be fixed so that $\Omega_g$ is a domain of class $C$. {Moreover, assume that   $g>0$ a.e.\,in $D \setminus \ogg$.}
Then, 
{$$S_\eps(g){|_{\Omega_g}} \weakly y_g \quad \text{in }H^1(\O_g) \ \text{ as }\eps \searrow 0,$$}where $S_\eps$ is the control-to-state map associated to \eqref{eq} and  $y_g \in H^1_0(\O_g)$ is the unique solution to 
\begin{equation}\label{awa}
  \begin{aligned}
   -\laplace y + \beta(y)&=f \quad \text{in }\O_g,
   \\y&=0  \quad \text{on } \partial \O_g.
 \end{aligned}
\end{equation}
\end{proposition}
\begin{remark}\label{rem:fs}
All the admissible domains $\O_g \in \OO$ (Definition \ref{o}) satisfy (together with their associated parametrization) the hypotheses of Proposition \ref{prop:cor}, \ro{as a result of Lemma \ref{fs}}.
\end{remark}
\begin{proof}
We follow the arguments from the proof of \cite[Thm.\,1]{nt}. For the beginning, let $\eps>0$ be arbitrary but fixed. We abbreviate for simplicity $y_\eps:=S_\eps(g)$ and, by multiplying \eqref{eq} with $y_\eps$, we infer
  \begin{equation}\label{e1}
  \|y_\eps\|_{\hd}^2 +\int_D \beta(y_\eps)y_\eps \,dx+\frac{1}{\eps}\int_D H_\eps(g)y_\eps^2 \,dx \leq \int_D (f +\eps g)y_\eps \,dx.\end{equation}
 Since $\beta$ is monotonically increasing and $H_\eps(g)\geq 0,$ by assumption, we  deduce from \eqref{e1} the estimate
  $$\|y_\eps\|_{\hd}^2 +\underbrace{\int_D [\beta(y_\eps)-\beta(0)]y_\eps }_{\geq 0}\,dx\leq \int_D (f +\eps g)y_\eps \,dx-\int_D \beta(0)y_\eps \,dx.
  $$By applying Young's inequality on the right-hand side, we derive uniform bounds w.r.t.\,$\eps$ so that we can extract a weakly convergent subsequence    \begin{equation}\label{yw}y_\eps \weakly \widehat y\quad \text{in }\hd.  \end{equation}
  In view of \eqref{eq:flip} and the compact embedding $\hdd \embed \embed \ld$, we then have 
  \begin{equation}\label{b2}
  \beta(y_\eps) \to \beta(\widehat y) \quad \text{in }\ld;
\end{equation}note that $\|y_\eps\|_{\li} \leq c,$ where $c>0$ is independent of $\eps,$ by Lemma \ref{lem:S}. 
Further, multiplying \eqref{e1} with $\eps$ implies
  $$\int_D H_\eps(g)y_\eps^2 \,dx \to 0 \text{ as }\eps \searrow 0.$$
  In light of \eqref{reg_h}, it holds $H_\eps(g) \to H(g) \ \text{in }L^q( D),\ q\in [1,\infty),$ where $H:\R \to \{0,1\}$ stands for the Heaviside function \eqref{h}.
Thus, by \eqref{yw} and the compact embedding $\hdd \embed \embed L^q(D),$ one obtains
  $$\int_{D\setminus  \ogg} \widehat y^2 \,dx =\int_D H(g)\widehat y^2 \,dx  = \lim_{\eps \searrow 0} \int_D H_\eps(g)y_\eps^2 \,dx = 0 ,$$
where the first identity is due to \eqref{def:og}, \eqref{h}, and $g>0$ a.e.\,in $D \setminus \ogg$,  by assumption. Hence, $\widehat y=0\ \ae D\setminus  \ogg$. As $\widehat y \in \hd$, cf.\,\eqref{yw}, it can be extended by zero on $\R^2$ (while preserving the $H^1$ regularity). Since   $\O_g$ is of class $C$, by assumption, the result in \cite[Thm.\,2.3]{t_c} implies that the weak limit from \eqref{yw} satisfies   \begin{equation}\label{y_h01og}
\widehat y \in H^1_0(\O_g).
\end{equation}

Testing \eqref{eq} with $\phi \in C_c^\infty(\O_{g})$ further  implies 
$$\int_{\O_{g}} \nabla y_\eps \nabla \phi \,dx +\int_{\O_{g}} \beta(y_\eps) \phi \,dx =\int_{\O_{g}} (f +\eps g)\phi \,dx,$$ since $H_\eps(g)=0$ on $ \O_{g}$ (cf.\,definition \eqref{def:og} and \eqref{reg_h}). Passing to the limit $\eps \searrow 0$, where one relies on \eqref{yw}, \eqref{b2}, then results in
$$\int_{\ogg} \nabla \widehat y \,\nabla \phi \,dx +\int_{\ogg} \beta(\widehat y) \phi \,dx =\int_{\ogg} f \phi \,dx.$$Since $\widehat y \in H^1_0(\O_g),$ cf.\,\eqref{y_h01og}, we have $\widehat y_{|\O_g}=y_g$ and the proof is now complete.
\end{proof}

The  result in Proposition \ref{prop:cor} allows us to introduce
\begin{definition}[The control-to-state map associated to the state equation in \eqref{p_shh}]\label{S}
We define
\begin{equation}
\SS:g \in \FF_{\mathfrak{s}} \mapsto y_g \in  H^1_0(\O_g) \cap H^2(\O_g), \end{equation}
where $y_g$ solves the equation \eqref{awa} on the component  of $\O_g$ containing $E$, i.e., $\O_g \in \OO$ (with a little abuse of notation, we do not make a difference between the notation for $\O_g$ and the notation for its relevant component; we just write $\O_g \in \OO$ when we mean the latter, cf.\,Remark \ref{rem:e}).
\end{definition}

Note that the additional $H^2(\O_g)$ regularity is due to the fact that each admissible shape $\O_g$ is of class $C^2$ (Lemma \ref{fs}.\ref{class}) and \cite[Lem.\,9.17]{gt}.

\begin{remark}\label{rem:S}
According to Definition \ref{S},  $\SS(g)$ exists only as an element of $H_0^1(\O_g).$ Whenever we write $\SS( g)$ as an element of $\hdd$ in what follows,  we think of its extension by zero outside $\O_g \in \OO$.
\end{remark}
\subsection{Convergence properties}\label{sec:s_e}
This subsection is dedicated to the study of  other limit behaviours of solutions to \eqref{eq} when $\eps \to 0$. In addition to the convergence from Proposition \ref{prop:cor}, we will need in the proof of our main result (Theorem \ref{thm:cor})  two convergence results that are contained in Lemmas \ref{lem:cp} and \ref{lem:convv} below. Note that, from now on, we simply write $\O_g \in \OO$ when we talk about the relevant component of $\O_g$ (Remark \ref{rem:e}).
\begin{lemma}\label{lem:cp}
Let $\{g_\e\} \subset \FF$ and $g \in \FF_{\mathfrak{s}}$ so that 
\[g_\e \to g \quad \text{in }\ld \cap L^\infty(D \setminus \bar E)\quad \text{as }\e \searrow 0.\]
Then, for each compact subset $K$ of $\ogg$, 
there exists $\e_0>0$, independent of $x$, so that
\begin{equation}\label{g_k}
g_\e \leq 0 \quad  \ae K,\ \forall\,\e\in(0,\e_0].\end{equation}
Moreover,
\[S_\eps( g_\e) \weakly  \SS( g) \quad \text {in }\hd \quad \text{as }\e \searrow 0.\]
\end{lemma}
\begin{proof}
%
Since $g$ is continuous, we have
\[g(x)\leq - \delta \quad \forall\, x \in K \subset \subset \Omega_g\]
for some $\delta>0.$
As $g_\e \to g $ in $L^\infty(D \setminus \bar E)$, by assumption, there exists $\e>0$, small, independent of $x$, so that
\[g_\e(x) \leq g(x)+\delta/2 \leq -\delta/2<0 \quad \ae K \cap (D \setminus \bar E).\]
Since $g_\e \in \FF$, we can thus conclude \eqref{g_k}.

Let us now show the desired convergence. We abbreviate for simplicity $y_\eps:=S_\eps(g_\e)$ and, by arguing exactly as in the proof of Proposition \ref{prop:cor}, we show that 
there exists a  weakly convergent subsequence (denoted by the same symbol) with  \begin{equation}\label{yw0}y_\eps \weakly \widehat y\quad \text{in }\hd \quad \text{as }\eps \searrow 0, \end{equation}
  \begin{equation}\label{b20}
  \beta(y_\eps) \to \beta(\widehat y) \quad \text{in }\ld \quad \text{as }\eps \searrow 0.
\end{equation}
Moreover,
 \begin{equation}\label{heg}
 \int_D H_\eps(g_\e)y_\eps^2 \,dx \to 0 \text{ as }\eps \searrow 0. \end{equation}
  We also notice that, by Lemma \ref{lem:S}, there exists a constant $C>0$, independent of $\eps$, so that
 \begin{equation}\label{ye0}
 \|y_\eps\|_{L^\infty(D)}\leq C \quad \forall \eps>0.
 \end{equation}
  As a result of $g_\e \to g$ in $\ld$ combined with Lemma \ref{lem:app}, we  have 
\[\mu\{x \in D:g<0 \text{ and }g_\e \geq 0\} \to 0 \quad \text{as }\e \searrow 0,\]
\[\mu\{x \in D:g>0 \text{ and }g_\e-\e \leq 0\} \to 0 \quad \text{as }\e \searrow 0\]
and with \eqref{ye0}, it follows
\[\lim_{\e \to 0}  \int_{\{g<0, g_\e\geq 0\}} H_\eps(g_\e)y_\eps^2 \,dx+\lim_{\e \to 0}  \int_{\{g>0,g_\e\leq \e\}} H_\eps(g_\e)y_\eps^2 \,dx=0.\]
  Thus, by \eqref{heg}, Lemma \ref{fs}.\ref{ls},  \eqref{reg_h} and \eqref{yw0} combined with the compact embedding $H^1(D) \embed \embed \ld$, one deduces
\begin{equation}\label{h_y}
\begin{aligned}
0=\lim_{\e \to 0} \int_D H_\eps(g_\e)y_\eps^2 \,dx &=\lim_{\e \to 0}  \int_{\{g<0, g_\e<0\}} H_\eps(g_\e)y_\eps^2 \,dx+\lim_{\e \to 0}  \int_{\{g>0,g_\e>\e\}} H_\eps(g_\e)y_\eps^2 \,dx
\\&=\lim_{\e \to 0}  \int_{\{g>0,g_\e>\e\}} y_\eps^2 \,dx=\lim_{\e \to 0}  \int_{\{g>0\}} y_\eps^2 \,dx
=\lim_{\e \to 0}  \int_{D\setminus \ogg} y_\eps^2 \,dx
\\&=\int_{D\setminus  \O_{g}}  \widehat y^2 \,dx .\end{aligned}
\end{equation}
%
Hence, $\widehat y=0\ \ae D\setminus   \O_{g}$, and  in view of \cite[Thm.\,2.3]{t_c} (applied for the relevant component of $\O_g$), this implies that  $$\widehat y \in H^1_0( \O_{g}).$$Note that here we use that $ \O_{g} \in \OO$ is a domain of class $C$ (even $C^2$, see Lemma \ref{fs}.\ref{class}).
 Testing \eqref{eq} w.r.h.s.\,$g_\e$ with $\phi \in C_c^\infty( \O_{g}),  \O_{g} \in \OO,$ further yields
$$\int_{\O_{g}} \nabla y_\eps \nabla \phi \,dx +\int_{\O_{g}} \beta(y_\eps) \phi \,dx +\int_{ \O_{g}}\frac{1}{\eps}H_\eps( g_\eps)y_\eps \phi \,dx=\int_{\O_{g}} (f +\eps g_\e)\phi \,dx.$$ For a fixed $\phi \in C_c^\infty( \O_{g})$ there exists a compact subset $\widetilde K$ of $\ogg$ so that 
$\phi \in C_c^\infty(\widetilde K)$. Hence, by \eqref{g_k} and \eqref{reg_h}, the third term in the above variational identity vanishes for $\e>0$ small enough, independent of $x$ (dependent on $\widetilde K$, and thus on $\phi$). 
Passing to the limit $\eps \searrow 0$, where one relies on \eqref{yw0}, \eqref{b20}, and the uniform boundedness of $\{g_\e\}$ in $\ld$, then results in
$$\int_{\ogg} \nabla \widehat y \,\nabla \phi \,dx +\int_{\ogg} \beta(\widehat y) \phi \,dx =\int_{\ogg} f \phi \,dx.$$As $\widehat y \in H^1_0(\O_g),$ and since \eqref{awa} is uniquely solvable, we conclude  $\widehat y=\SS(g)$ and, thanks to \eqref{yw0}, the proof is now complete.
\end{proof}
\begin{assumption}\label{yd_f}
For  the given datum $f$ and the desired state $ y_d$ we require
 \begin{equation}\label{f}
f \geq \beta(0)\quad \ae D,\quad y_d \leq 0\quad \ae E\end{equation}or, alternatively, 
\begin{equation}\label{f0}
f \leq \beta(0) \quad \ae D,\quad y_d \geq 0\quad \ae E.\end{equation}
\end{assumption}
\begin{lemma}\label{lem:convv}
If Assumption \ref{yd_f} is satisfied, 
 then
 \[\liminf_{\eps \to 0} \int_E (S_\eps(  g_\eps)(x)-y_d(x))^2 \; dx \geq \liminf_{\eps \to 0} \int_E (\SS(  g_\eps)(x)-y_d(x))^2 \; dx \]for each sequence $\{g_\e\}\subset \FF_{\mathfrak{s}}$ that is uniformly bounded in $\ld$.
\end{lemma}
\begin{proof}Let $\e>0$ be arbitrary but fixed. Arguing as in Lemma \ref{lem:S} we see that the equation 
\begin{equation}\label{ye}
  \begin{aligned}
   -\laplace y_\e + \beta(y_\e)+\frac{1}{\eps}H_\e(g_\e)y_\e&=f\quad \text{in }D,
   \\y_\e&=0\quad \text{on } \partial D
 \end{aligned}
\end{equation}admits a unique solution $y_\e:=\widehat S_\eps( g_\eps) \in \hd \cap H^2(D)$. Since $\{g_\e\}$ is uniformly bounded in $\ld$, by assumption, we have 
\[\|y_\e- S_\eps( g_\eps)\|_{\hd} \leq \e\|g_\e\|_{\ld} \to 0 \quad \text{as }\e \searrow 0.\]Thus, it suffices to show  
\begin{equation}\label{ye4}
\liminf_{\eps \to 0} \int_E (y_\eps(x)-y_d(x))^2 \; dx \geq \liminf_{\eps \to 0} \int_E (\SS(  g_\eps)(x)-y_d(x))^2 \; dx. \end{equation}
We abbreviate  $z_\e:=\SS(g_\eps)$, see Definition \ref{S} and \eqref{awa}. In view of Lemma \ref{lem:g>0}, we can define  $\hge \in \FF_{\mathfrak{s}}$ so that the entire set $\O_{\hge}$ is just   the relevant component of $\O_{g_\e}$. Hence, $z_\e$ is the unique solution of 
\begin{equation}
  \begin{aligned}
   -\laplace z_\e + \beta(z_\e)&=f \quad \text{in }\O_{\hge},
   \\z_\e&=0  \quad \text{on } \partial \O_{\hge}.
 \end{aligned}
\end{equation}
Let $k>0$ be  arbitrary, but fixed  and define $z_\e^k:=\widehat S_k (\widetilde g_\e)$, i.e., $z_\e^k$ solves 
\begin{equation}\label{yee}
  \begin{aligned}
   -\laplace z_\e^k + \beta(z_\e^k)+\frac{1}{k}H_k(\widetilde g_\e)z_\e^k&=f \quad \text{in }D,
   \\z_\e^k&=0\quad \text{on } \partial D.
 \end{aligned}
\end{equation}
In view of Lemma \ref{fs}, the hypotheses of Proposition \ref{prop:cor} are satisfied. By arguing in the exact same way as in the proof thereof we deduce that there exists a (not relabelled) subsequence $\{z_\e^k\}_k$ so that
\begin{equation}\label{zek}
z_\e^k \weakly z_\e \quad \text{in }H_0^1(D)\quad \text{as }k \to 0.\end{equation}With a little abuse of notation, $z_\e$ denotes here its extension by zero on the whole domain $D$.
\\(I) In the case when \eqref{f} is true, 
it holds \begin{equation}\label{zek0}
y_\e \geq 0 , \ z_\e^k \geq 0,\  z_\e \geq 0 \quad \ae D .\end{equation}
 Our next aim is   to compare $z_\e$ with $y_\e$. To this end, we rewrite \eqref{yee} as
\begin{equation}\label{yee0}
  \begin{aligned}
   -\laplace z_\e^k + \beta(z_\e^k)+\frac{1}{\e}H_\e(g_\e)z_\e^k&=f +\underbrace{\Big[\frac{1}{\e}H_\e(g_\e)-\frac{1}{k}H_k(\widetilde g_\e)\Big]}_{\leq 0}\underbrace{z_\e^k}_{\geq 0}\quad \text{in }D,
   \\z_\e^k&=0\quad \text{on } \partial D.
 \end{aligned}
\end{equation}
We observe that, by \eqref{reg_h}, it holds $H_k \geq H_\e$ for $k \leq \e$, which implies
\[\frac{1}{\e}H_\e(g_\e) \leq \frac{1}{k}H_k(g_\e) \leq \frac{1}{k}H_k(\widetilde g_\e) \quad \text{a.e.\,in }D,\quad \forall\,0<k\leq\eps.\] Note that in the last inequality we employed the monotonicity of $H_k$, cf.\,\eqref{reg_h}, and the fact that  $g_\e \leq \widetilde g_\e$, by construction, see Lemma \ref{lem:g>0}. 
Hence, comparing \eqref{ye} with \eqref{yee0} yields 
\[z_\e^k \leq y_\e    \quad \ae D,\quad \forall\,0<k\leq\eps.\] By passing to the limit $k \to 0$, where we use \eqref{zek}, we  arrive at 
\[z_\e \leq y_\e    \quad \ae D.\]Then, \eqref{zek0} and $y_d \leq 0 \ \ae E$, see \eqref{f}, lead to
\[0\leq z_\e -y_d \leq y_\e -y_d\quad  \ae E,\]
whence \eqref{ye4} follows.
\\(II) The case when \eqref{f0} in Assumption \ref{yd_f}  is true can be  treated completely analogously. One obtains
\[
y_\e \leq 0 , \ z_\e^k \leq 0,\  z_\e \leq 0 \quad \ae D,\]
\[z_\e \geq y_\e    \quad \ae D.\] Then, $y_d \geq 0 \ \ae E$, cf.\ \eqref{f0}, gives in turn \[0\geq z_\e -y_d \geq y_\e -y_d \quad  \ae E.\]
 The proof is now complete.
\end{proof}

\subsection{Existence of optimal solutions for \eqref{p10}}\label{ex_p10}
In the proof of our main result (Theorem \ref{thm:cor}), we will need that the approximating control problem \eqref{p10} admits optimal solutions (when the admissible set $\FF$ is restricted to a certain ball). For this, we gather some useful convergences in the next lemma. Throughout this subsection, $\eps>0$ is arbitrary but fixed, and the solution operator to \eqref{eq} is denoted by $S_\eps$, cf.\, Lemma \ref{lem:S}.

\begin{lemma}\label{wc}
Let $\{g_n\} \subset  \FF$ with \[g_n \weakly g \quad \text{in } \WW,\text{ as }n \to \infty.\] Then, 
\begin{equation}\label{he_conv}
H_\eps (g_{n}) \to H_\eps(g) \quad \text{in }L^q(D),\ q \in[1,\infty) \ \text{ as }n \to \infty,\end{equation}
\begin{equation}\label{wc_S}
S_\eps(g_n) \weakly S_\eps(g) \quad \text{in }H^2(D) \cap \hd, \text{ as }n \to \infty.\end{equation}
\end{lemma}
\begin{proof}We begin by observing that,  due to $\{g_n\} \subset  \FF$, we have  $g\in \FF$, and thus, it holds
\begin{equation}\label{chi}
H_\eps (g_{n})=0=H_\eps(g) \quad \ae E, \ \forall\,n.
\end{equation} In light of the compact embedding $\hde \embed \embed L^2(D \setminus E),$ and the continuity of $H_\eps$, we also have \[H_\eps (g_{n}) \to H_\eps(g) \quad \text{in }L^2(D\setminus E), \ \text{ as }n \to \infty.\]This implies a.e.\,convergence on a subsequence and since $H_\e(g_n)\in [0,1]$, we may now deduce  \eqref{he_conv}, by Lebesgue's dominated convergence theorem. To show \eqref{wc_S}, we abbreviate for simplicity $y_n:=S_\eps(g_n)$. As a result of Lemma \ref{lem:S}, the sequence $\{y_n\}$ is uniformly bounded in $H^2(D) \cap \hd,$ and we can thus extract a weakly convergent subsequence denoted by the same symbol so that
\[y_n \weakly \widetilde y \quad \text{in }H^2(D) \cap \hd \ \text{ as }n \to \infty.\]
 Passing to the limit $n \to \infty$ in \eqref{eq}, where one uses \eqref{he_conv}, then yields
\begin{equation}\label{eq_n}
  \begin{aligned}
   -\laplace \widetilde y + \beta(\widetilde y)+\frac{1}{\eps}H_\eps(g) \widetilde y&=f +\eps g\quad \text{a.e.\,in }D,
   \\\widetilde y&=0\quad \text{on } \partial D.
 \end{aligned}
\end{equation}From Lemma \ref{lem:S} we finally deduce 
$\widetilde y=S_\eps(g) \in H^2(D) \cap \hd$  and the proof of \eqref{wc_S} is  complete.\end{proof}

By means of the direct method of calculus of variations, we can now show the following result.
\begin{proposition}\label{ex_e}
The approximating optimal control problem \eqref{p10} admits at least one global minimizer in $\FF$.
\end{proposition}

\section{Correlation between \eqref{p10} and shape optimization}\label{cor}
This section contains the main results of the present paper. One of them states  that local optima in the $\ld-$sense of \eqref{p_shh} (Definition \ref{def_sh}) can be approximated by local minimizers of the approximating control problems \eqref{p10} introduced in the last section. As a consequence, we deduce that each parametrization of the optimal shape of \eqref{p_sh} is the limit of global optima of \eqref{p10} (Corollary \ref{cor:os}). These two findings  allowed us to continue the research on the topic of optimality systems for \eqref{p_sh} in the recent work \cite{p3}.

Similarly to Definition \ref{def_sh}, we introduce the notion of local optimality for the approximating control problem.
\begin{definition}\label{def_e}Let $\e>0$ be fixed and $\bar g_{\mathfrak{s}} \in \FF_{\mathfrak{s}}$.
We say that $\bar g_\eps \in \FF$ is locally optimal for \eqref{p10} in the $\ld$ sense, if there exists $r>0$ such that 
\begin{equation}\label{loc_opt_e}
j_\eps(\bar g_\eps) \leq j_\eps(g) \quad \ff g \in \FF \text{ with }\|g-\bar g_\e\|_{\ld}\leq  r,
\end{equation}where $$j_\eps(g):=\int_E (S_\eps( g)(x)-y_d(x))^2 \; dx+\alpha\,\int_{D} (1-H_\eps(g))(x)\,dx+\frac{1}{2}\,\|g-\bar g_{\mathfrak{s}}\|_{\WW}^2$$ is the reduced cost functional associated to the control problem \eqref{p10}.
\end{definition}
\begin{theorem}\label{thm:cor}
Let $\bar g_{\mathfrak{s}} \in \FF_{\mathfrak{s}}$ be a local minimizer of \eqref{p_shh} in the sense of Definition \ref{def_sh}. Then, under Assumption \ref{yd_f},  there exists a sequence of local minimizers $\{\bar g_\eps\}$ of \eqref{p10} such that 
\[\bar g_\eps \to \bar g_{\mathfrak{s}} \quad \text{in }\WW, \text{ as }\eps \searrow 0.\]{Moreover,}\[S_\eps(\gbe) \weakly \SS(\bar g_{\mathfrak{s}})\quad \text{in }\hd, \text{ as }\eps \searrow 0.\]
\end{theorem}
\begin{proof}
The  proof is a variation of well-known ideas \cite{barbu84}. However, let us point out that we need to take into account that the control-to-state map $\SS$ is defined on $\FF_{\mathfrak{s}}$, not on $\FF$.  To bridge the gap between these two sets, we will make use of the essential Proposition \ref{dense} and the convergence results from subsection \ref{sec:s_e}.

Denote by $\overline{ B_{L^2(D)}(\bar g_{\mathfrak{s}},r)}$, $r>0,$ the (closed) ball of local optimality of $\bar g_{\mathfrak{s}} \in \FF_{\mathfrak{s}}$. Let $\e>0$ be fixed. Thanks to Lemma \ref{wc}, we can show by the direct method of calculus of variations that the optimization problem 
\begin{equation}\label{min_je}
\min_{g \in  \FF \cap \overline{ B_{L^2(D)}(\bar g_{\mathfrak{s}},r/2})} j_\eps(g)
\end{equation}
 admits a global minimizer 
 \[\bar g_\eps \in \FF \cap \overline{ B_{L^2(D)}(\bar g_{\mathfrak{s}},r/2)}.\]
 We recall here that
 $\WW$ is the Hilbert space $\ld \cap \hde$, $s\in(1,2]$, endowed with the norm 
\[\|\cdot\|_{\WW}^2:=\|\cdot\|_{\ld}^2+\|\cdot\|_{H^s(D\setminus \bar E)}^2,\]
and 
\begin{equation}
 \FF:=\{g \in \WW: g \leq 0\text{ a.e.\,in }E\},
\end{equation}
cf.\,\eqref{f_tilde}. Note that the set $ \FF \cap \overline{ B_{L^2(D)}(\bar g_{\mathfrak{s}},r/2)}$ is weakly closed, i.e., for $\{z_n\}\subset  \FF \cap \overline{ B_{L^2(D)}(\bar g_{\mathfrak{s}},r/2)}$ we have 
\[z_n \weakly z \quad \text{in }\WW \Rightarrow z \in \FF\cap \overline{ B_{L^2(D)}(\bar g_{\mathfrak{s}},r/2)}. \]
 Next, we recall that the reduced cost functional associated to the control problem \eqref{p_shh} is given by
 $$\JJ(g)=\int_E (\SS( g)(x)-y_d(x))^2 \; dx+\alpha\,\int_{D} (1-H(g))(x)\,dx, \quad g \in \FF_{\mathfrak{s}}.$$ 
 Due to Proposition \ref{prop:cor} (see also Remark \ref{rem:fs}),  it holds 
  \begin{equation}\label{f_e}\begin{aligned}
\JJ(\bar g_{\mathfrak{s}})&=\lim_{\eps \to 0} \int_E (S_\eps( \bar g_{\mathfrak{s}})(x)-y_d(x))^2 \; dx+\alpha\,\int_{D} (1-H_\eps(\bar g_{\mathfrak{s}}))(x)\,dx
\\&\quad \geq \limsup_{\eps \to 0} \int_E (S_\eps( \gbe)(x)-y_d(x))^2 \; dx+\alpha\,\int_{D} (1-H_\eps(\gbe))(x)\,dx
\\&\qquad +\frac{1}{2}\,\|\gbe-\bar g_{\mathfrak{s}}\|_{\WW}^2, \end{aligned} \end{equation}\normalsize where for the last inequality we relied on the fact that $\bar g_\eps$ is a global minimizer of \eqref{min_je} and that $\bar g_{\mathfrak{s}}$ is admissible for \eqref{min_je}.
Thanks to Proposition \ref{dense}, for each $\eps>0$, there exists $\widehat g_\eps \in \FF_{\mathfrak{s}}$ so that 
\begin{equation}\label{ge_d}
\|\widehat g_\eps-\gbe\|_{\ld}\leq \frac{r\eps}{2}\min\{ \frac{1}{L_{H_\eps}},\frac{1}{L_{\|\bar g_{\mathfrak{s}}\|_{\ld}+r,\eps}}\}\leq r/2,
\end{equation}
where $L_{H_\eps}>0$ and $L_{\|\bar g_{\mathfrak{s}}\|_{\ld}+r,\eps}>0$ are the Lipschitz constant of $H_\eps$ and $S_\eps$ (see \eqref{slip}), respectively; note that the second estimate in \eqref{ge_d} is true for $\eps$ small enough. 
Then,
\begin{equation}\label{he_c}
\|H_\eps(\widehat g_\eps)-H_\eps(\gbe)\|_{\ld}\leq r\eps/2,\end{equation}
and according to \eqref{slip}, we also have
\begin{equation}\label{se_c}\|S_\eps(\widehat g_\eps)-S_\eps(\gbe)\|_{\ld}\leq L_{\|\bar g_{\mathfrak{s}}\|_{\ld}+r,\eps}\|\widehat g_\eps-\gbe\|_{\ld}\leq r\eps/2. 
\end{equation}
Here we note that \begin{equation}\label{gehat}
\|\widehat g_\eps\|_{\ld},\|\gbe\|_{\ld}\leq \|\bar g_{\mathfrak{s}}\|_{\ld}+r,\end{equation}since $\gbe$ is admissible for \eqref{min_je}  and in light of \eqref{ge_d}.


Going back to \eqref{f_e}, we can thus write 
  \begin{equation}\begin{aligned}\label{11}
\JJ(\bar g_{\mathfrak{s}})&=\lim_{\eps \to 0} \int_E (S_\eps( \bar g_{\mathfrak{s}})(x)-y_d(x))^2 \; dx+\alpha\,\int_{D} (1-H_\eps(\bar g_{\mathfrak{s}}))(x)\,dx
\\&\geq \limsup_{\eps \to 0} \int_E (S_\eps( \gbe)(x)-y_d(x))^2 \; dx+\alpha\,\int_{D} (1-H_\eps(\gbe))(x)\,dx
\\&\quad +\frac{1}{2}\,\|\gbe-\bar g_{\mathfrak{s}}\|_{\WW}^2
\\&\geq \liminf_{\eps \to 0} \int_E (S_\eps( \gbe)(x)-y_d(x))^2 \; dx+\alpha\,\int_{D} (1-H_\eps(\gbe))(x)\,dx
\\&\quad +\frac{1}{2}\,\|\gbe-\bar g_{\mathfrak{s}}\|_{\WW}^2
\\&= \liminf_{\eps \to 0} \int_E (S_\eps( \widehat g_\eps)(x)-y_d(x))^2 \; dx+\alpha\,\int_{D} (1-H_\eps(\widehat g_\eps))(x)\,dx
\\&\quad +\frac{1}{2}\,\|\gbe-\bar g_{\mathfrak{s}}\|_{\WW}^2,
 \end{aligned} \end{equation}
where we relied on \eqref{he_c} and \eqref{se_c}.
 Thanks to  Lemma \ref{lem:convv} combined with \eqref{gehat} and  $H_\e\leq H$ (see \eqref{reg_h} and \eqref{h}), \eqref{11} can be continued as 
    \begin{equation}\begin{aligned}\label{111}
\JJ(\bar g_{\mathfrak{s}})&\geq \liminf_{\eps \to 0} \int_E (\SS( \widehat g_\eps)(x)-y_d(x))^2 \; dx+\alpha\,\int_{D} (1-H(\widehat g_\eps))(x)\,dx
\\&\quad +\frac{1}{2}\,\|\gbe-\bar g_{\mathfrak{s}}\|_{\WW}^2
\\&\geq  \int_E (\SS( \bar g_{\mathfrak{s}})(x)-y_d(x))^2 \; dx+\alpha\,\int_{D} (1-H(\bar g_{\mathfrak{s}}))(x)\,dx
\\&=\JJ(\bar g_{\mathfrak{s}}).
 \end{aligned} \end{equation}
  The last estimate in \eqref{111} is due to the fact that $\widehat g_\eps \in \FF_{\mathfrak{s}}$ is in the ball of local optimality of $\bar g_{\mathfrak{s}}$ (in light of \eqref{ge_d} and since $\gbe$ is admissible for \eqref{min_je}).
 From \eqref{11} and \eqref{111} we can now conclude \begin{equation}\label{conv_gn}
 \begin{aligned}
\lim_{\eps \to 0} &\int_E (S_\eps( \gbe)(x)-y_d(x))^2 \; dx+\alpha\,\int_{D} (1-H_\eps(\gbe))(x)\,dx
+\frac{1}{2}\,\|\gbe-\bar g_{\mathfrak{s}}\|_{\WW}^2
\\&=\JJ(\bar g_{\mathfrak{s}})=\int_E (\SS( \bar g_{\mathfrak{s}})(x)-y_d(x))^2 \; dx+\alpha\,\int_{D} (1-H(\bar g_{\mathfrak{s}}))(x)\,dx
\\&=\lim_{\eps \to 0} \int_E (S_\eps( \gbe)(x)-y_d(x))^2 \; dx+\alpha\,\int_{D} (1-H_\eps(\gbe))(x)\,dx,\end{aligned}\end{equation}where the last identity follows by estimating as in \eqref{11}-\eqref{111}, without taking the term $\frac{1}{2}\,\|\gbe-\bar g_{\mathfrak{s}}\|_{\WW}^2$ into account. By \eqref{conv_gn}, we finally deduce 
\begin{equation}\label{gn}
\gbe \to \bar g_{\mathfrak{s}} \quad \text{in }\WW.\end{equation}
To show that $\gbe$ is  a local minimizer of \eqref{p10}, let $v \in \FF$ with $\|v-\gbe\|_{\ld}\leq r/4$ be arbitrary, but fixed. Then, by \eqref{gn}, we have for $\e>0$ small enough
\[\|v-\bar g_{\mathfrak{s}}\|_{\ld} \leq \|v-\gbe\|_{\ld}+\|\gbe-\bar g_{\mathfrak{s}}\|_{\ld}\leq r/2.\]As $\gbe$ is global optimal for \eqref{min_je}, this means that $j_\e(\gbe) \leq j_\e(v)$, whence the desired local optimality of $\gbe$ follows (Definition \ref{def_e}).

Since $s>1$, we have $\WW \embed \ld \cap L^\infty(D \setminus \bar E)$, and Lemma \ref{lem:cp} with \eqref{gn} implies \[S_\eps(\gbe) \weakly \SS(\bar g_{\mathfrak{s}})\quad \text{in }\hd, \text{ as }\eps \searrow 0.\]
The proof is now complete.\end{proof}

\begin{remark}[Approximation of global optima of \eqref{p_shh}]
A short  inspection of the proof of Theorem \ref{thm:cor} shows that, if $\bar g_{\mathfrak{s}}$ is a global minimizer, then the associated sequence $\{\bar g_\e\}$ inherits this property, i.e., it consists of global optima of \eqref{p10}.
Here we pay attention to the fact that \eqref{gehat} is replaced by
\begin{equation}\label{gehat'}
\|\widehat g_\eps\|_{\ld},\|\gbe\|_{\ld}\leq c,\end{equation}where $c>0$ is a constant independent of $\eps$ (this can be proven by arguing as in the proof of \eqref{f_e}).
Thus, we can still make use of Lemma \ref{lem:convv} to deduce  \eqref{111}.
\end{remark}

As  a consequence of Theorem \ref{thm:cor} and Proposition \ref{rem:equiv} we then arrive at the following

\begin{corollary}[Approximation of  the optimal shape]\label{cor:os}
Suppose that  Assumption \ref{yd_f} is satisfied. 
Let $\O^\star \in \OO$ be an optimal shape for \eqref{p_sh}. Then, for each  $g^\star \in \FF_{\mathfrak{s}}$ with $\O_{g^\star}=\O^\star$, there exists  
a sequence of global minimizers $\{\bar g_\eps\}$ of \eqref{p10} such that 
\[\bar g_\eps \to g^\star \quad \text{in }\WW, \text{ as }\eps \searrow 0.\]
{Moreover,}\[S_\eps(\gbe) \weakly \SS(g^\star)\quad \text{in }\hd, \text{ as }\eps \searrow 0.\]
\end{corollary}

\section*{Acknowledgment}
This work was financially supported by the DFG grant BE 7178/3-1 for the project "Optimal Control of Viscous
Fatigue Damage Models for Brittle Materials: Optimality Systems".  
\\The author thanks  Prof.\,Dan Tiba (Institute of Mathematics of the Romanian Academy) for  useful discussions during the writing of this
paper.

\appendix \section{A useful convergence}
\begin{lemma}\label{lem:app}Let $\MM$ be a measurable set and assume that 
$h_n \to h$ in $L^1(\MM)$ as $n \to \infty$. Then,
\[\mu\{x\in \MM:h(x)>0 \text{ and }h_n(x)\leq 0\} \to 0 \quad \text{ as }n \to \infty.\]
\end{lemma}
\begin{proof}
The proof is inspired by the proof of \cite[Lem.\,A.2]{hmw13} and uses the result established in \cite[Lem.\,A.1]{hmw13}. 
We abbreviate $\MM^+:=\{x \in \MM:h(x)>0\}$ and $\MM_n:=\{x\in \MM:h(x)>0 \text{ and }h_n(x)\leq 0\}.$ If the desired assertion is not true, then there exists $\gamma>0$ and a subsequence $\{n_k\}$ so that
\[\mu(\MM_{n_k}) \geq \gamma \quad \forall\, k\in \N.\]From \cite[Lem.\,A.1]{hmw13} we know that
\[\int_{\MM_{n_k}} h \,dx \geq \delta \quad \forall\, k\in \N\]for some $\delta>0.$
Further, it holds
\begin{align*}\|\max\{h_{n_k},0\}- \max\{h,0\}\|_{L^1(\MM^+)}&\geq \int_{\MM_{n_k}} |\max\{h_{n_k},0\}- \max\{h,0\}| \,dx
\\&\quad =\int_{\MM_{n_k}} h \,dx \geq \delta \quad \forall\, k\in \N, \end{align*}
where, for the identity, we made use of the definition of $\MM_n$. Since $h_n \to h$ in $L^1(\MM)$, by assumption, we get a contradiction and the proof is now complete.
\end{proof}

\bibliographystyle{plain}
\bibliography{strong_stat_coupled_pde}

\end{document}